\documentclass{amsart}[14pt]
\usepackage{amsmath}
\usepackage{graphicx}
\usepackage[dvipsnames]{xcolor}
\usepackage{hyperref}
\usepackage{amssymb,bbold}
\usepackage{enumitem}
\usepackage{manfnt}
\hypersetup{colorlinks=true,citecolor=blue,linkcolor=cyan}
\newtheorem{thm}{Theorem}[section]
\newtheorem{cor}[thm]{Corollary}

\newtheorem{lem}[thm]{Lemma}
\newtheorem{conj}[thm]{Conjecture}

\newtheorem{prop}[thm]{Proposition}
\theoremstyle{definition}

\theoremstyle{remark}
\newtheorem{rem}[thm]{Remark}
\numberwithin{equation}{section}

\newcommand{\F}{\mathbb{F}}

\newcommand{\C}{\mathbb{C}}

\newcommand{\Q}{\mathbb{Q}}

\newcommand{\Tr}{\mbox{Tr}}

\newcommand{\rad}{\mbox{rad}}
\newcommand{\Res}{\mbox{Res}}
\newcommand{\cond}{\mbox{cond}}

\def\cL{\mathcal{L}}
\def\cM{\mathcal{M}}
\def\cK{\mathcal{K}}
\def\cH{\mathcal{H}}

\def\cZ{\mathcal{Z}}

\def\cD{\mathcal{D}}

\def\fM{\mathfrak{M}}
\def\fH{\mathfrak{H}}
\def\fP{\mathfrak{P}}

\newcommand{\ccom}[1]{{\color{red}{Chantal: #1}} }

\newcommand{\kommentar}[1]{}

\newlist{primenumerate}{enumerate}{1}
\setlist[primenumerate,1]{label={(\arabic*$'$})}
\begin{document}

\title[Expected Values]{Expected Values of cubic Dirichlet $L$-functions Away from the Central Point}%

\author{Chantal David}
\address{Chantal David: 
Department of Mathematics and Statistics, Concordia University, 1455 de Maisonneuve West, Montr\'eal, Qu\'ebec, Canada H3G 1M8}
\email{chantal.david@concordia.ca}
\author{Patrick Meisner}
\address{Patrick Meisner: Department of Mathematical Sciences, Chalmers University of Technology, SE-412 96 Gothenburg, Sweden}
\email{meisner@chalmers.se}



\begin{abstract}

We compute the expected value of Dirichlet $L$-functions over $\mathbb{F}_q[T]$ attached to cubic characters evaluated at an arbitrary $s \in (0,1)$.
We find a transition term at the point $s=\frac{1}{3}$, reminiscent of the transition at the point $s=\frac{1}{2}$ of the bound for the size of an $L$-function implied by the Lindel\"of hypothesis. We show that at $s=\frac{1}{3}$, the expected value matches corresponding statistics of the group of unitary matrices multiplied by a weight function. 
This is the first result in the literature computing the first moment at $s=\tfrac13$ for any family of cubic Dirichlet characters, over function fields or number fields, and 
it involves the deep connections between Dirichlet series of cubic Gauss sums and metaplectic Eisenstein series first introduced by Kubota, 
which is necessary to obtain the cancellation between the principal sum and the dual sum occurring at  $s=\tfrac13$.
\end{abstract}
\maketitle

\section{Introduction}

\subsection{Setup and Main Result}

Let $q=p^a$ be a prime power and consider the ring of polynomials $\mathbb{F}_q[T]$ consisting of polynomials with coefficients in the finite field $\mathbb{F}_q$. Denote $\chi$ as a Dirichlet character of $\mathbb{F}_q[T]$ and $\mathcal{M}$ the set of monic polynomials in $\mathbb{F}_q[T]$. Then define the $L$-function attached to $\chi$ as
\begin{align}\label{L-FuncDef}
    L(s,\chi) = \sum_{F\in \mathcal{M}} \frac{\chi(F)}{|F|^s},
\end{align}
where $|F| := q^{\deg F}$.

The Riemann Hypothesis implies that there exists a conjugacy class of unitary matrices $\Theta_\chi$, called the Frobenius class, such that 
\begin{align}\label{FrobDef}
L(s,\chi) = (1-q^{-s})^{1-\delta_\chi}\det(1-q^{\frac{1}{2}-s}\Theta_\chi)
\end{align}
where $\delta_\chi=0$ or $1$ depending on the parity of $\chi$. From here we can impose some trivial bound on the size of $L(s,\chi)$, and we have as $q$ tends to infinity
\begin{align}\label{transition}
    |L(s,\chi)| \ll \begin{cases} q^{(\frac{1}{2}-s)g} & s<\frac{1}{2} \\ 2^g & s=\frac{1}{2} \\ 1 & s>\frac{1}{2}
    \end{cases}
\end{align}
where $g$ is the dimension of $\Theta_\chi$. That is, it will be bounded for $s$ above $\frac{1}{2}$, while the bound grows exponentially as $s$ decreases below $\frac{1}{2}$. The bound at $s=\frac{1}{2}$ depends only on the genus and is believed to grow polynomially in $g$ on average.

Conrey, Farmer, Keating, Rubinstein and Snaith \cite{CFKRS} developed a recipe for conjecturing moments of $L$-functions over $\mathbb{Q}$ at $s=\frac{1}{2}$. Andrade and Keating \cite{AK} then adapt the recipe to $L$-functions over $\mathbb{F}_q(T)$ and conjecture that
\begin{align}\label{QuadMome}
    \frac{1}{|\mathfrak{M}_2(2g)|} \sum_{\chi\in\mathfrak{M}_2(2g)}L({\textstyle{\frac12}},\chi)^k = P_k(g)+o(1)
\end{align}
where $\mathfrak{M}_2(2g)$ is the set of primitive quadratic characters of genus $2g$ and $P_k$ is an explicit polynomial of degree $\frac{k(k+1)}{2}$. They prove this for $k=1$, while Florea \cite{F1,F2,F3} confirms it in the cases $k=2,3,4$, improving the quality of error terms from the number field case (for $k=4$, only the leading three terms of the polynomial $P_k$ of degree 10 can be obtained). 
In a recent spectacular work of Bergstr\"om, Diaconu, Petersen, Westerland \cite{BDPW} and Miller, Patzt, Petersen, Randal-Williams \cite{MPPR}, the authors show that for any positive integer $k$, the conjectures of \cite{CFKRS} hold for the $k$-th moment of quadratic characters over $\F_q[T]$, provided that $q$ is large enough with respect to $k$  (and where $q > 2^{24(k+1)}$ is enough!)

In addition to the papers mentionned above, many papers have been published on moments of quadratic characters at the central point both in number fields and function fields, including  \cite{AK,D,DW,Jut,KeS,Sound,Young}.

Moments of higher order characters have been less studied however, there has been some recent interest, including \cite{Baier-Young, Di2004,FHL,Luo,DFL}). In particular, the first author, with Florea and Lalin \cite{DFL}, showed that as $g\to\infty$
\begin{align}\label{DFLThm}
    \frac{1}{|\mathfrak{M}^1_3(g)|} \sum_{\chi \in \mathfrak{M}^1_3(g)} L(\textstyle{\frac12},\chi) = A_q + o(1)   
\end{align}
where $A_q$ is an explicit constant, depending only on $q$, and $\mathfrak{M}^1_3(g)$ is the set of primitive cubic characters of genus $g$ with the added condition that if $q\equiv 1 \mod{3}$, then $\chi|_{\mathbb{F}_q^*} = \chi_3$, a fixed, nontrivial cubic character of $\mathbb{F}_q^*$. Notice  that the expected value at $s=\frac12$ is bounded, contrary to the case of quadratic characters. This can be explained by the fact that quadratic characters form a symplectic family, while cubic characters form a unitary family.

In this paper, we are interested in what happens to values such as in \eqref{DFLThm} for arbitrary values of $0<s<1$ and for the family
\begin{align}\label{fHg}
\fH(g) = \left\{ \chi_F := \left(\frac{\cdot}{F}\right)_3 : \mbox{ genus of $\chi_F$ is $g$ and $F$ is square-free} \right\}
\end{align}
where $\left(\frac{\cdot}{F}\right)_3$ is the cubic-residue symbol modulo $F$ (as defined in Section \ref{characters}). 


\kommentar{Note that $\fH(g)$ is a thin subset of all cubic characters,
and (as we will see in Section \ref{characters}), that if $q\equiv 1 \mod{3}$ then $\fH(3g) \subset \mathfrak{M}^1_3(3g)$ while $\fH(3g+2)=\emptyset$. Further, the characters of $\fH(3g+1)$ will be of mixed parity (odd or even).}

In the following theorem, and in all the results of this paper, the error terms are independent of $q$, unless stated otherwise. 

\begin{thm}\label{MainThm} 
Let  $q\equiv 1 \bmod{6}$ be a prime power,  and $\epsilon>0$.  Then for any $0<s<1, s \neq \frac{1}{3}$, we have 
$$\frac{1}{|\fH(3g)|} \sum_{\chi\in \fH(3g)} L(s,\chi) = M_q(s) + O_{\epsilon} \left( q^{\frac{3}{10}(1-6s+\epsilon)g} +  q^{-\frac15 (1 + 9s - 15 \varepsilon )g + 2-2s} + q^{(1-3s)g}E_s(g) \right)$$
where
$$M_q(s) = \zeta_q(3s)\prod_P \left(1-\frac{1}{|P|^{3s}(|P|+1)}\right), \;\;\;
E_s(g) =  \begin{cases} 1 & s< \frac{2}{3} \\ g^2 & s=\frac{2}{3} \\ gq^{\frac{6}{5}(3s-2)g} & s>\frac{2}{3}\end{cases}$$
and $\zeta_q(s)$ is defined by \eqref{def-zeta}.\\
If $s=\frac{1}{3}$, we have $$\frac{1}{|\fH(3g)|} \sum_{\chi\in\fH(3g)} L({\textstyle{\frac13}},\chi)  = C_q (g +B_q) + O (1),$$
where the $O(1)$ is absolute (it does not depend on $q$ or $g$) as long at $g \geq 2$, $B_q$ is defined by \eqref{def-Bq}, and 
$$C_q = \prod_P \left(1-\frac{1}{|P|(|P|+1)}\right).$$
\end{thm}

\begin{cor} \label{coro} Let  $q\equiv 1 \bmod{6}$ be a fixed prime power. For  $\frac13 < s < 1$, and $g \rightarrow \infty$, we have 
\begin{align*}
\frac{1}{|\fH(3g)|} \sum_{\chi\in \fH(3g)} L(s,\chi) &\sim M_q(s) .
\end{align*}
For $s= \frac13$, and $g \rightarrow \infty$, we have 
 \begin{align*}
\frac{1}{|\fH(3g)|} \sum_{\chi\in\fH(3g)} L({\textstyle{\frac13}},\chi) &\sim C_q  \, g.
\end{align*}
\end{cor}

To our knowledge, this is the first result in the literature computing the first moment at $s=\tfrac13$ for any family of cubic Dirichlet characters, over function fields or number fields. The fact that there is a ``transition" at $s=\frac13$ for the size of the moments of $L(s, \chi)$ for cubic characters $\chi$ (and at $s=\frac{1}{\ell}$ for  the moments of $L(s, \chi)$ for characters of order $\ell$) was predicted in \cite{FHL}, but without computing those moments.
Based on the random matrix model of Section \ref{RMT-model}, we would also predict the same transition for the first moment at $s=\tfrac{1}{\ell}$ for  families of characters of order $\ell$. We remark that there  is no hope to prove such a result for $\ell > 3$ with the current knowledge, as only very partial results are known for the residues of the Dirichlet series of $\ell$-th order Gauss sums, and the explicit computation of those residues is crucial for the proof of Theorem \ref{MainThm}.
 

\begin{rem} 
When $s=\frac12$, we recover a result similar to the main results of \cite{DFL}, but for a different family, as they considered the full families of cubic characters over $\F_q(t)$. 
 The full families are significantly more challenging that the thin family of \eqref{fHg}: in the non-Kummer case, the relevant Gauss sums are not defined  over $F_q(t)$, which leads to unbalanced sums between $\F_q(t)$ and $\F_{q^2}(t)$, and in the Kummer case, the family is ``too large". In both cases, getting an error term good enough to see the cancellation between the principal and the dual sum is much more challenging. In \cite{DFL}, the authors uncover a similar cancellation of the secondary term of the principal sum and the dual sum at for the moment at $s=\frac12$ (compare Propositions \ref{PrincComp} and \ref{DualComp}), but the cancellation occurred inside the error term, and then computing the moment at $s=\frac13$ for the full families is much more challenging.

\end{rem}


Comparing Theorem \ref{MainThm} to \eqref{transition}, we see a similar phenomenon occurring, except with a transition at $s=\frac{1}{3}$. When $s>\frac{1}{3}$, the error term decays with $g$ and we get an explicit constant. At $s=\frac{1}{3}$, the error term stops decaying, but the main term has a pole, resulting in a contribution of roughly a constant times one-third of the genus (since the characters in $\fH(3g)$ have genus  $3g$).  When $s<\frac{1}{3}$, the error term begins growing exponentially, making the main term no longer a main term.

The main technique in proving Theorem \ref{MainThm} is to use the approximate functional equation to write the $L$-function as a principal sum and a dual sum, and we average each of them over the family. In typical applications where one computes the moments at $s=\frac{1}{2}$, the main term will come from the principal sum and the oscillations of the sign of the functional equation will make the dual sum smaller. 

In fact, for $s>\frac{1}{3}$, we find that the principal sum has $M_q(s)$ as a main term as well as a secondary term that decays. However, this secondary term has a pole at $s=\frac{1}{3}$ and starts growing exponentially for $s<\frac{1}{3}$. While for the dual sum we find a main term which is exactly the negative of the secondary term of the principal sum. Analyzing the poles at $s=\frac{1}{3}$ of each of the principal and dual sum then yields the result for $s=\frac{1}{3}$.

The principal sum is a straightforward double character sum and we apply standard number theory techniques to compute it. The dual sum comes with an extra factor due to the sign of the functional equation, and we then need to compute an  average of cubic Gauss sums. This relies on the deep connection between Dirichlet series of cubic Gauss sums and  metaplectic Eisenstein series first introduced by Kubota \cite{Ku1969, Ku1971}. Explicit computations for the residues of the Dirichlet series where then done by Patterson \cite{P1978} (see also \cite{KP1984}), and over function fields, by Hoffstein \cite{Hoffstein} and Patterson \cite{P}. Some precise  formulas that are needed to evaluate our averages of cubic Gauss sums were developed in \cite{DFL}, building on \cite{Hoffstein, P}. We will extend the main result of \cite{DFL} in order to compute averages of Gauss sums weighted by Euler products (see Proposition \ref{DFLPropExt}).



Finally, we would like to bring attention to the transition of $E_s(g)$ at $s=\frac{2}{3}$. This is completely analogous with what happens to the main term at $s=\frac{1}{3}$ and is not too surprising since the functional equation relates the value at $\frac{1}{3}$ to the value at $\frac{2}{3}$. With this, it seems reasonable that the true behaviour of $E_{\frac{2}{3}}(g)$ should be a constant times $g$, and not $g^2$. However, it is not immediately clear what to expect for the true behaviour beyond $\frac{2}{3}$.

It is natural to ask if Theorem \ref{MainThm} would hold (unconditionally) over number fields, i.e averaging over a similar thin family of cubic characters over $\Q(\xi_3)$, where $\xi_3$ is a primitive third root of unity, and this is studied in the forthcoming work of Hamdar. 

\subsection{Random Matrix Model} \label{RMT-model}

For statistics at $s=\frac{1}{2}$, the $q$-dependence on the right-hand side of \eqref{FrobDef} disappears (when $\delta_{\chi}=1$), and
\begin{align*}
    L(\textstyle{\frac12},\chi) = \det(1-\Theta_\chi).
\end{align*}
That is, we may apply the Katz-Sarnak \cite{KaS} philosophy which states that the Frobenii of a family of $L$-functions should equidistribute in some compact matrix Lie group.

Specifically, if $f$ is any continuous class function and $\mathcal{F}_g$ is a ``nice'' family of $L$-functions of fixed genus $g$, then the Katz-Sarnak philosophy predicts that there is some compact matrix Lie group, $G_g$\footnote{ $G_g$ will be a subgroup of $U(g)$ and will typically be one of $U(g)$, $USp(g)$, $SO(g)$, $SO_{even}(g)$, $SO_{odd}(g)$.}, such that
\begin{align}\label{KSlim}
    \lim_{q\to\infty} \frac{1}{|\mathcal{F}_g|} \sum_{L\in \mathcal{F}_g} f(\Theta_L) = (1+o(1))\int_{G_g} f(U)dU.
\end{align}
where $dU$ is the Haar measure and the $o(1)$ term vanishes as $g$ tends to infinity.

Applying this with the continuous class function $f(U) = \det(1-U)^k$ allows us to predict that the moments of $L$-functions in ``nice'' families at $s=\frac{1}{2}$ should behave like the moments of the characteristic polynomial of a matrix in a compact matrix Lie group at $1$, and
\begin{align*}
    \lim_{q\to\infty} \frac{1}{|\mathcal{F}_g|} \sum_{L\in \mathcal{F}_g} L({\textstyle{\frac12}})^k = (1+o(1))\int_{G_g} \det(1-U)^k dU.
\end{align*}

This framework can further help to explain the recipe in \cite{CFKRS} and the ensuing conjecture in \eqref{QuadMome}. That is, it is known that the Frobenii attached to quadratic characters are symplectic and that
$$\int_{USp(2g)} \det(1-U)^k dU = Q_k(g)$$
for some explicit polynomials of degree $\frac{k(k+1)}{2}$, which coincides with the conjecture of Andrade and Keating \cite{AK}.

Many results (\cite{BCDGL,CP,DG,M3,M1}) suggest that the compact matrix Lie group attached to cubic characters is the group of unitary matrices. This is further reinforced by \eqref{DFLThm} and Theorem \ref{MainThm} with the observation that
$$\lim_{q\to\infty} A_q = \lim_{q\to\infty} M({\textstyle{\frac12}}) = \int_{U(g)} \det(1-U)du = 1.$$

However, this framework can not, a priori, help to explain the transition term at $s=\frac{1}{3}$ since the right hand side of
\begin{align}\label{Lat1/3}
L({\textstyle{\frac13}},\chi) = \det(1-q^{\frac16}\Theta_\chi)
\end{align}
is not independent of $q$ and, hence, can not be identified with a \textit{single} continuous class function $f$ of the unitaries as $q$ grows.

\subsection{Weighted Random Matrix Model}

An important set of continuous class functions are the mixed trace functions
$$P_{\lambda}(U) := \prod_{j=1}^{\infty} \Tr(U^j)^{\lambda_j}$$
where $\lambda = 1^{\lambda_1}2^{\lambda_2}\cdots$ is the partition consisting of $\lambda_1$ ones, $\lambda_2$ twos, etc. These form a basis for continuous class functions of the unitary matrices and so it would be enough to prove \eqref{KSlim} for $f=P_{\lambda}$ for all $\lambda$.

Towards this, the second author in \cite{M2} proved a partial result towards \eqref{KSlim}  
for some family of cubic characters $\mathcal{F}_3(N)$. Specifically, for all $\lambda$ such that $|\lambda| := \sum j\lambda_j < N$ we have
\begin{align}\label{UnitaryFamily}
    \lim_{q\to\infty} \frac{1}{|\mathcal{F}_3(N)|} \sum_{\chi\in \mathcal{F}_3(N)} P_{\lambda}(\Theta_\chi) = \int_{U(N)} P_\lambda(U) dU.
\end{align}

One unsatisfying aspect of \eqref{UnitaryFamily} is that the right hand side is always $0$. The main goal of \cite{M2} was to find an appropriate normalization for the left hand side to be non-zero and see how that could effect the random matrix interpretation. Indeed, it was shown that for $|\lambda|<\frac{3N}{4}$ \cite[Theorem 1.1 with $r=3$]{M2}
\begin{align}\label{WeightedUnitaryMatrix}
    \lim_{q\to\infty} \frac{1}{|\mathcal{F}_3(N)|} \sum_{\chi\in\mathcal{F}_3(N)} q^{\frac{|\lambda|}{6}} P_{\lambda}(\Theta_\chi) = \int_{U(N)} P_{\lambda}(U) \overline{\det(1-\wedge^3U)}dU
\end{align}
where
\begin{align*}
\det(1-\wedge^3U) := \prod_{1 \leq i_1 < i_2 <  i_3 \leq N} \left( 1 - x_{i_1} x_{i_2} x_{i_3} \right)
\end{align*}
and the $x_i$ are the eigenvalues of $U$.

While extending \eqref{UnitaryFamily} to all $\lambda$ would give us a result like \eqref{KSlim} it is not immediately clear what extending \eqref{WeightedUnitaryMatrix} for all $\lambda$ would give us. One possible interpretation is that, by the definition of $P_\lambda$, we have $q^{\frac{|\lambda|}{6}}P_{\lambda}(\Theta_\chi) = P_{\lambda}(q^{\frac{1}{6}}\Theta_\chi)$ and so an extension of \eqref{WeightedUnitaryMatrix} would imply the following conjecture.

\begin{conj}\label{Conjecture}
Let $\mathcal{F}(N)$ be a ``nice" family of cubic characters defined over $\mathbb{F}_q[T]$ with conductor of degree $N$, and let $f$ be a continuous class function. Then
$$\lim_{q\to\infty} \frac{1}{|\mathcal{F}(N)|}\sum_{\chi \in \mathcal{F}(N)} f(q^\frac{1}{6}\Theta_\chi) = (1+o(1))\int_{U(N)} f(U)\overline{\det(1-\wedge^3U)}dU$$
where the $o(1)$ tends to $0$ as $N$ tends to infinity.
\end{conj}

Notice that the result of \cite{M2} implies that Conjecture \ref{Conjecture} is true for the family of $\mathcal{F}_3(N)$ and all continuous class functions in the span of $\{P_{\lambda} : |\lambda|<3N/4\}$. Similar techniques would prove similar results for the family $\fH(3g)$ as well. However, we see by \eqref{Lat1/3} that
$$L({\textstyle{\frac13}},\chi) = f(q^{\frac{1}{6}}\Theta_\chi) $$
where $f(U)=\det(1-U)$ is not in the span of $\{P_{\lambda} : |\lambda|<3N/4\}$. Regardless of this, we can use Theorem \ref{MainThm} to prove Conjecture \ref{Conjecture} in this case.

\begin{thm}\label{RMTThm}
    Conjecture \ref{Conjecture} is true with the family $\fH(3g)$ when $g \geq 2$, and the continuous class function $f(U) = \det(1-U)$.  Specifically,
    \begin{align*}
    \lim_{q\to\infty}  \frac{1}{|\fH(3g)|} \sum_{\chi\in \fH(3g)} L({\textstyle{\frac13}},\chi)  &= g + O(1), \\
    \end{align*}
    and \begin{align*}
    \int_{U(3g)} \deg(1-U) \overline{\det(1-\wedge^3U)} dU &= g + 1.
    \end{align*}
\end{thm}


\subsection{Structure of the paper} We present in Section \ref{background} some background on L-functions and cubic characters over function fields, and we describe the thin family of cubic characters that we are using. In particular, we generalize the approximate functional equation of \cite{DFL} from $s=\frac12$ to general $s$. Using the approximate functional equation, the average value can be written as a principal sum and a dual sum as in \eqref{AFEApplied}. We compute the principal sum in Section \ref{section-principal}, the dual sum in Section \ref{section-dual-sum}, and using those estimates, the proof of theorems \ref{MainThm} and \ref{RMTThm} are given in Section \ref{proof-thms}. 

\textbf{Acknowledgements}: The authors would like to thank the anonymous referee for several suggestions and corrections that
improved the exposition of the paper. CD was supported by the Natural Sciences and Engineering Research Council
of Canada [DG-155635-2019] and by 
the Fonds de recherche du Qu\'ebec Nature et technologies [Projet de recherche en \'equipe
300951]. PM was supported by the grant KAW 2019.0517 from the Knut and Alice Wallenberg Foundation. 

\section{Background on $L$-functions} \label{background}

\subsection{Cubic characters over $\F_q[T]$} \label{characters}

We denote by $\cM$ the set of monic polynomials in $\F_q[T]$, and by $\cM_d$, repectively $\cM_{\leq d}$,  the set of monic polynomials in $\F_q[T]$ of degree $d$, respectively of degree $\leq d$.

Let $q$ be a prime power with $q \equiv 1 \mod 6$. We fix once and for all an isomorphism $\Omega$ from the cube roots of unity in $\F_q^*$ to $\mu_3 = \{ 1, \xi_3, \xi_3^2 \}$, the cube roots of unity in $\C^*$, where $\xi_3 = e^{2 \pi i/3}$. We then denote by $\chi_3$ the cubic residue symbol of $\F_q^*$ given by 
\begin{align} \label{def-chi3}
\chi_3(a) = \Omega(a^{(q-1)/3}), \;\; \text{for all $a \in \F_q^*$}.
\end{align}
For each prime $P \in \F_q[T]$, we define the cubic residue symbol of conductor $P$
$$\chi_{P} : \F_q[T]/(P) \rightarrow \mu_3
$$
as follows: for $a \in \F_q[T]$, if $P \mid a$, then $\chi_P(a)=0$, and otherwise, $\chi_P(a) = \alpha$, where $\alpha \in \mu_3$ is such that
$$a^{\frac{q^{\deg P}-1}{3}} \equiv \Omega^{-1}(\alpha) \bmod P.$$

There are then 2 cubic characters of conductor $P$, $\chi_P$ and $\chi_P^2 = \overline{\chi}_P$.
We extend the definition to $F \in \cM$ by multiplicativity. Writing $F = P_1^{e_1} \dots P_s^{e_s}$ where the $P_i$ are distinct primes and the $e_i$ are positive integers, we define
\begin{align} \label{by-mult}
\chi_{F}(a)  = \chi_{P_1}(a)^{e_1} \dots \chi_{P_s}(a)^{e_s}.
\end{align}
Then, $\chi_F$ is a cubic character of conductor $\rad(F) = P_1 \dots P_s$. Conversely, all the primitive cubic characters of conductor $P_1 \dots P_s$ are given by
$\chi_{P_1}^{e_1} \dots \chi_{P_s}^{e_s}$ with $1 \leq e_i \leq 2$, and there are $2^s$ such characters.

We say that a cubic character $\chi$ is even if $\chi \vert_{\F_q^*} = \chi_0$, the trivial character, and that  $\chi$ is odd if $\chi \vert_{\F_q^*} = \chi_3$ or $\chi_3^2$. We define
$$
\delta_\chi = \begin{cases} 1 & \text{when $\chi$ is odd} \\
0 & \text{when $\chi$ is even.} 
\end{cases}$$

The best classification of cubic characters is by genus. 
From the Riemann-Hurwitz formula (Theorem 7.16 of \cite{Rosen}), we compute
$$g = \deg{\cond(\chi)} - 2 + \delta_\chi,$$
and we denote by 
$\fM_{3}(g)$ the set of primitive cubic characters over $\F_q[T]$ of genus $g$.
$\fM_{3}(g)$  is naturally divided in three disjoint subsets $\fM_{3}^0(g), \fM_{3}^1(g), \fM_{3}^2(g)$, depending on the restriction of $\chi$ over $\F_q$, and we define for $j=0,1,2$,
\begin{align} \label{Kummer-family}
\fM_{3}^j(g) &= \{\chi \in \fM_3(g) \;:\; \chi |_{\mathbb{F}_q^*} = \chi_3^j\} \end{align}
where we identify $\chi_3^0 = \chi_0$. In particular, we get
\begin{align*}
    \fM_3(g) = \fM_3^0(g) \cup \fM_3^1(g) \cup \fM_3^2(g).
\end{align*}
Using the observation that ${q^n-1} \equiv {n(q-1)}\mod{3}$, we get that if $a\in \mathbb{F}^*_q$, then for any prime $P$,
$$a^{\frac{q^{\deg P}-1}{3}} = a^{\frac{ (q-1)}{3} \deg{P}} \equiv  \Omega^{-1}\left( \chi^{\deg P}_3(a)\right) \mod{P}.$$
Extending this multiplicatively, we find that $\chi_F|_{\mathbb{F}_q^*} = \chi_3^{\deg F}.$
Hence, if we define
$$\cM_3^j(d) := \{F\in \F_q[T] : F \mbox{ cube-free}, \deg{\rad(F)}=d, \deg F \equiv j \bmod{3}\},$$
where $\rad{(F)} = \prod_{P \mid F} P$, 
then we have 
\begin{align*} 
\fM_{3}^0(g)  &=  \{ \chi_F \; : \; F\in \cM_3^0(g+2) \}
\end{align*}
and for $j=1,2$, 
\begin{align*} 
\fM_{3}^j(g)  =  \{ \chi_F \; : \;  F  \in \cM_3^j(g+1) \}
\end{align*}
Let
$$\mathcal{H}(d) := \{F \in \mathbb{F}_q[T] : F \mbox{ square-free}, \deg F =d\}.$$
For $g\equiv 0 \bmod{3}$, let
$$\fH(g) := \{ \chi_F : F\in \cH(g+1)\} \subset \fM_3^1(g) \subset \fM_3(g) $$
while if $g\equiv 1 \bmod{3}$, let $$\fH(g) := \{ \chi_F : F\in \cH(g+2)\cup \cH(g+1) \} \subset \fM_3^0(g)\cup\fM_3^2(g) \subset \fM_3(g).$$
Somewhat surprisingly, if $g\equiv 2 \bmod{3}$, then we find that there are no elements of $\fM_3(g)$ with a square-free discriminant. Hence we set $\fH(g) = \emptyset$ in this case.

We see that there is a  natural bijection from $\fH(3g)$ to $\mathcal{H}(3g+1)$, which is reminiscent of the family usually considered for quadratic characters $\mathfrak{M}_2(2g)$ which comes with a natural bijection to $\mathcal{H}(2g+1)$. From this point of view,
the family $\fH(3g)$ is a ``natural" extension of the quadratic family.

As the functional equation of the $L$-functions depends on the parity of the character, it will be useful to distinguish them. Thus, we define
\begin{align*}
    \fH_e(g) := \{ \chi \in \fH(g) : \chi \mbox{ is even} \} \quad \quad \mbox{ and } \quad \quad 
    \fH_o(g) := \{ \chi \in \fH(g) : \chi \mbox{ is odd} \}.
\end{align*}
Notice that the even characters are exactly those in $\fM_3^0(g)$ so that if $g\not\equiv 1 \bmod{3}$, $\fH_e(g) = \emptyset$.

Now, the standard square-free sieve tells us that for $d\geq 1$
\begin{align} \label{SF-poly}
|\cH(d)| = \frac{q^d}{\zeta_q(2)} = q^d-q^{d-1}, \end{align}
where $\zeta_q(s)$ is defined by \eqref{def-zeta}, which gives
\begin{align*}
    |\fH_e(g)| = \begin{cases} 0 & g \not \equiv 1 \bmod{3} \\ |\cH(g+2)| & g \equiv 1 \bmod{3} \end{cases} = \begin{cases} 0 & g \not \equiv 1 \bmod{3} \\ q^{g+2}-q^{g+1} & g \equiv 1 \bmod{3} \end{cases}
\end{align*}
and
\begin{align*}
    |\fH_o(g)| = \begin{cases} |\cH(g+1)| & g  \not \equiv 2 \bmod{3} \\0 & g \equiv 2 \bmod{3} \end{cases} = \begin{cases} q^{g+1}-q^g & g  \not\equiv 2 \bmod{3}  \\ 0 & g \equiv 2 \bmod{3} \end{cases}
\end{align*}
from which we may conclude that
$$|\fH(g)| = |\fH_e(g)|+|\fH_o(g)| = \begin{cases}  q^{g+1}-q^g & g \equiv 0 \mod{3} \\ q^{g+2}-q^g & g \equiv 1 \mod{3} \\ 0 & g\equiv 2 \mod{3} \end{cases}.$$

\subsection{Functional Equation}

The affine zeta function over $\F_q[T]$ is defined by
\begin{align} \label{def-Z}
\cZ_q(u) = \sum_{f \in \cM} u^{\deg{f}} = \prod_P \left( 1 - u^{\deg{P}} \right)^{-1} = \frac{1}{(1 - qu)} 
\end{align}
for $|u|<q^{-1}$. The right-hand side provides an analytic continuation to the entire complex plane, with a simple pole at $u=q^{-1}$ with residue $-\tfrac{1}{q}$. We also define
\begin{align} \label{def-zeta}
\zeta_q(s) = \cZ_q(q^{-s}) = (1-q^{1-s})^{-1}. \end{align}
Replacing in \eqref{SF-poly}, we express the size of $\fH(g)$ in terms of values of $\zeta_q(s)$:
$$|\fH(g)| =  \begin{cases}  \frac{q^{g+1}}{\zeta_q(2)} & g \equiv 0 \mod{3} \\ \frac{q^{g+2}}{\zeta_q(3)} & g \equiv 1 \mod{3} \\ 0 & g\equiv 2 \mod{3}. \end{cases}$$
Let $\chi$ be a primitive cubic Dirichlet character as defined in Section \ref{characters}, and let $h \in \cM$ be its conductor. We define the $L$-function in the $u$-variable as
$$\cL(u,\chi) = \sum_{F\in \mathcal{M}} \chi(F) u^{\deg F}$$
so that $L(s,\chi) = \cL(q^{-s},\chi)$, where $L(s,\chi)$ is defined in \eqref{L-FuncDef}. If $\chi$ is even, we have that $\cL(1,\chi)=0$, and 
we define the completed $L$-function
\begin{align}\label{even-odd}
    \mathcal{L}_C(u,\chi) = \frac{\cL(u,\chi)}{(1-u)^{1-\delta_\chi}}.
\end{align}
Let $g$ be the genus of the character $\chi$.
It follows from the Weil conjectures \cite{W} that
$\mathcal{L}_C(u,\chi)$ is a polynomial of degree $g$ and it satisfies the functional equation
\begin{align}\label{FE}
    {\cL_C}(u,\chi) = \omega(\chi) (\sqrt{q}u)^g \mathcal{L}_C\left(\frac{1}{qu},\overline{\chi}\right)
\end{align}
where $\omega(\chi)$ is the sign of the functional equation. To give a formula for $\omega(\chi)$, we need to define the Gauss sums of characters over $\F_q[T]$.

We first start with Gauss sums for characters over $\F_q^*$. If $\chi$ is a non-trivial character of $\F_q$, 
we define
\begin{align*}
\tau(\chi) &:= \sum_{a\in \mathbb{F}_q^*} \chi(a) e^{\frac{2\pi i tr_{\mathbb{F}_q/\mathbb{F}_p} (a)} {p}} .
\end{align*}
Then, $\tau(\overline{\chi}) = \overline{\tau(\chi)}$ and $|\tau(\chi)| = q^{1/2}$, and we denote the sign on the Gauss sum by
\begin{align*}
\epsilon(\chi)  &:=  q^{-1/2} \tau(\chi).
\end{align*}
We extend the definition to trivial characters by defining $\epsilon(\chi_0)=1$. 


To define the Gauss sums of general characters over $\F_q[T]$, we define the exponential over $\F_q(T)$ as follows:
for any $a \in \F_q((1/T))$, we have
$$
e_q(a) = e^{\frac{2\pi i tr_{\mathbb{F}_q/\mathbb{F}_p} (a_1)} {p}},
$$
where $a_1$ is the coefficient of $1/T$ in the Laurent expansion of $a$. We then have the usual properties: $e_q(a+b) = e_q(a) e_q(b)$ and $e_q(a)=1$ for $a \in \F_q[T]$. Also, if $a,b,h \in \F_q[T]$ with $a \equiv b \bmod h$, then $e_q(a/h) = e_q(b/h)$.
For $\chi$ a primitive cubic character of modulus $h$ over $\F_q[T]$, the Gauss sum of $\chi$ is
$$
G(\chi) = \sum_{a \bmod h} \chi(a) e_q\left( \frac{a}{h} \right).
$$
It is not hard to show that $G(\overline{\chi}) = \overline{G(\chi)}$ and $|G(\chi)| = q^{\deg{h}/2}$.

\begin{lem}{\cite[Corollary 2.3]{DFL}}  \label{DFLLem1} Let $\chi$ be a primitive cubic character of conductor $h$. Then,
$$\omega(\chi) = \overline{\epsilon(\chi_3^{\deg{h}})} \; \frac{G(\chi)}{q^{\deg{h}/2}}$$
where $\chi_3$ is the cubic character defined in \eqref{def-chi3}.
We then have
\begin{align*}
\omega(\chi) = \overline{\omega(\chi)} \quad \quad \mbox{ and } \quad \quad |\omega(\chi)|=1.
\end{align*}
\end{lem}



The following result generalizes \cite[Proposition 2.4]{DFL} which gives the approximate functional equation when $s=\frac12$.
\begin{prop}[Approximate functional equation]
Let $\chi$ be a primitive character of genus $g$, and $A$ be a positive integer. If $\chi$ is odd, then
\begin{align} \label{oddAFE}
\cL(q^{-s}, \chi) =  \sum_{f\in \mathcal{M}_{\leq A}} \frac{\chi(f)}{|f|^s} + \omega(\chi) (q^{1/2-s})^{g} \sum_{f\in \mathcal{M}_{\leq g-A-1}} \frac{\overline{\chi}(f)}{|f|^{1-s}}
\end{align}
If $\chi$ is even, then
\begin{align}\label{evenAFE}
    \cL(q^{-s},\chi)  = & \frac{1}{1-q^{1-s}}\left[ \sum_{f\in \cM_{\leq A+1}} \frac{\chi(f)}{|f|^s} - q^{1-s} \sum_{f\in \cM_{\leq A}} \frac{\chi(f)}{|f|^s} \right] \nonumber \\
    & + \frac{1}{1-q^s} \frac{\omega(\chi)}{q^{(s-1/2)g}} \frac{\zeta_q(2-s)}{\zeta_q(s+1)} \left[ \sum_{f\in \cM_{\leq g-A}} \frac{\overline{\chi}(f)}{|f|^{1-s}} - q^{s} \sum_{f\in \cM_{\leq g-A-1}} \frac{\overline{\chi}(f)}{|f|^{1-s}} \right]
\end{align}

\end{prop}

\begin{proof}

Since $\cL(u,\chi)$ is a polynomial of degree $g+1-\delta_\chi$, we may write
\begin{align}\label{LPoly}
    \cL(u,\chi) = \sum_{n=0}^{g+1-\delta_\chi} a_n(\chi) u^n
\end{align}
where
\begin{align}\label{a_n-degf}
a_n(\chi) = \sum_{f\in \mathcal{M}_n} \chi(f).
\end{align}
Similarly, for $\cL_C(u,\chi)$ we write
\begin{align}\label{CompLPoly}
    \cL_C(u,\chi) = \sum_{n=0}^{g} b_n(\chi)u^n.
\end{align}
Substituting \eqref{CompLPoly} into \eqref{FE} and comparing coefficients we get 
\begin{align}\label{b_n-rel}
    b_n(\chi) = \omega(\chi)q^{n-g/2} b_{g-n}(\overline{\chi}).
\end{align}

Applying \eqref{even-odd}, we can write $a_n(\chi)$ in terms of $b_n(\chi)$ such that for $n=0,\dots,g$
\begin{align}\label{a_n-b_n-rel}
a_n(\chi) = \begin{cases} b_n(\chi) & \delta_\chi =1 \\ b_n(\chi)-b_{n-1}(\chi) & \delta_\chi=0  \end{cases}
\end{align}
while if $\delta_\chi=0$ then $a_{g+1}(\chi) = -b_{g}(\chi)$. Reversing \eqref{a_n-b_n-rel}, we can write $b_n(\chi)$ in term of $a_n(\chi)$ such that  
\begin{align}\label{b_n-def}
b_n(\chi) = \begin{cases} a_n(\chi) & \delta_\chi = 1 \\ \sum_{m=0}^n a_m(\chi) & \delta_\chi = 0 \end{cases}
\end{align}
for $n=0,\dots,g$.

Finally, if $\delta_\chi=0$, then we may apply \eqref{b_n-rel} to \eqref{a_n-b_n-rel} to obtain for any $0\leq A\leq g$,
\begin{align*}
    a_{g-A}(\chi) & = b_{g-A}(\chi)-b_{g-A-1}(\chi) \\
    & = \omega(\chi) q^{g/2-A}b_A(\overline{\chi}) - \omega(\chi)q^{g/2-A-1}b_{A+1}(\overline{\chi})\\
    & = \omega(\chi) q^{g/2-A}b_A(\overline{\chi}) - \omega(\chi)q^{g/2-A-1}\left(a_{A+1}(\overline{\chi}) + b_A(\overline{\chi})\right).
\end{align*}
Rearranging and taking conjugates we then obtain
\begin{align}\label{b_A-rel}
    \frac{b_A(\chi)}{q^{(A+1)s}} = \frac{1}{q-1}\left( \frac{\omega(\chi)}{q^{g/2}} a_{g-A}(\overline{\chi})q^{(A+1)(1-s)} + \frac{a_{A+1}(\chi)}{q^{(A+1)s}} \right)
\end{align}
where we have used $1/\omega(\overline{\chi}) = \omega(\chi)$.

Replacing $A$ with $g-A-1$, $s$ with $1-s$ and $\chi$ with $\overline{\chi}$ we obtain
\begin{align}\label{b_g-A-rel}
    \frac{b_{g-A-1}(\overline{\chi})}{q^{(g-A)(1-s)}} = \frac{1}{q-1}\left( \frac{\overline{\omega(\chi)}}{q^{g/2}} a_{A+1}(\chi)q^{(g-A)s} + \frac{a_{g-A}(\overline{\chi})}{q^{(g-A)(1-s)}} \right).
\end{align}

Splitting the sum in \eqref{CompLPoly} at an arbitrary point $A$, and applying \eqref{b_n-rel}, we obtain
\begin{align*}
    \mathcal{L}_C(u,\chi) & = \sum_{n=0}^A b_n(\chi)u^n + \sum_{n=A+1}^{g} b_n(\chi)u^n \\
    & = \sum_{n=0}^A b_n(\chi)u^n  + \omega(\chi) \sum_{n={A+1}}^{g} q^{n-g/2} b_{g-n}(\overline{\chi}) u^n \\
    & = \sum_{n=0}^{A} b_n(\chi) u^n + \omega(\chi) (\sqrt{q} u)^{g} \sum_{n=0}^{g-A-1} \frac{b_n(\overline{\chi})}{q^nu^n}.
\end{align*}

If $\delta_\chi=1$, then $\cL(u,\chi)=\cL_C(u,\chi)$ and $a_n(\chi)=b_n(\chi)$ and so we get the odd approximate function equation
\begin{align}\label{oddAFE}
    \cL(q^{-s},\chi) & = \sum_{n=0}^{A} \frac{a_n(\chi)}{q^{sn}} + \omega(\chi) (q^{1/2-s})^{g} \sum_{n=0}^{g-A-1} \frac{a_n(\overline{\chi})}{q^{(1-s)n}} \nonumber \\
    & = \sum_{f\in \mathcal{M}_{\leq A}} \frac{\chi(f)}{|f|^s} + \omega(\chi) (q^{1/2-s})^{g} \sum_{f\in \mathcal{M}_{\leq g-A-1}} \frac{\overline{\chi}(f)}{|f|^{1-s}}.
\end{align}

If $\delta_\chi=0$, then we get $\cL(u,\chi) = (1-u)\cL_C(u,\chi)$ and we get
\begin{align*}
    \cL(q^{-s},\chi) & = \sum_{n=0}^{A} \frac{b_n(\chi)}{q^{sn}}(1-q^{-s}) + \omega(\chi) q^{(1/2-s)g} \sum_{n=0}^{g-A-1} \frac{b_n(\overline{\chi})}{q^{(1-s)n}}(1-q^{-s}) \nonumber \\
    &  = \sum_{n=0}^{A} \frac{b_n(\chi)}{q^{sn}} \left ( 1- q^{-s} \right )  + \omega(\chi) \frac{\zeta_q(2-s)}{\zeta_q(s+1)} \frac{q^{g/2}}{q^{sg}}
    \sum_{n=0}^{g-A-1} \frac{b_n(\overline{\chi})}{q^{(1-s)n}} \left ( 1-q^{-(1-s)} \right ).
\end{align*}

Expanding the $(1-q^{-s})$ term in the first series and applying \eqref{a_n-b_n-rel} and \eqref{b_A-rel} with the observation that $b_0(\chi)=a_0(\chi)$, we obtain

\begin{align*}
\sum_{n=0}^{A} \frac{b_n(\chi)}{q^{sn}} \left ( 1- q^{-s} \right )  
&= b_0(\chi) + \sum_{n=1}^A \frac{b_n(\chi) - b_{n-1}(\chi)}{q^{sn}} - \frac{b_A(\chi)}{ q^{(A+1)s}}\\
& = b_0(\chi) + \sum_{n=1}^A \frac{a_n(\chi)}{ q^{sn}} - \frac{1}{q-1}\left( \frac{\omega(\chi)}{q^{g/2}} a_{g-A}(\overline{\chi})q^{(A+1)(1-s)} + \frac{a_{A+1}(\chi)}{q^{(A+1)s}} \right) \\
&=  \sum_{f \in \mathcal{M}_{\leq A}} \frac{\chi(f)}{|f|^s} - \frac{1}{q-1}\left( \frac{\omega(\chi)}{q^{g/2}} a_{g-A}(\overline{\chi})q^{(A+1)(1-s)} + \frac{a_{A+1}(\chi)}{q^{(A+1)s}} \right).
\end{align*}

Likewise, expanding the $(1-q^{-(1-s)})$ and applying \eqref{a_n-b_n-rel} and \eqref{b_g-A-rel}, the second series becomes 
\begin{align*}
    \sum_{f \in \mathcal{M}_{\leq g-A-1}} \frac{\overline{\chi}(f)}{ |f|^{1-s}} -\frac{1}{q-1}\left( \frac{\overline{\omega(\chi)}}{q^{g/2}} a_{A+1}(\chi)q^{(g-A)s} + \frac{a_{g-A}(\overline{\chi})}{q^{(g-A)(1-s)}} \right).
\end{align*}

Combining everything, when $\delta_\chi=0$, we obtain the even approximate function equation
\begin{align}\label{evenAFE}
    \cL(q^{-s},\chi) = & \sum_{f \in \mathcal{M}_{\leq A}} \frac{\chi(f)}{|f|^s} +  \frac{\zeta_q(2-s)}{\zeta_q(s+1)} \frac{\omega(\chi)}{q^{(s-1/2)g}} \sum_{f \in \mathcal{M}_{\leq g-A-1}} \frac{\overline{\chi}(f)}{ |f|^{1-s}} \nonumber \\
    & + \frac{1}{1-q^{1-s}} \frac{a_{A+1}(\chi)}{q^{(A+1)s}} + \frac{1}{1-q^{s}} \frac{\zeta_q(2-s)}{\zeta_q(s+1)}\frac{\omega(\chi)}{q^{(s-1/2)g}} \frac{a_{g-A}(\overline{\chi})}{q^{(g-A)(1-s)}}.
\end{align}
where we note that
$$-\frac{1}{q-1} \left(1  + \frac{\zeta_q(2-s)}{\zeta_q(s+1)}q^s\right) = \frac{1}{1-q^{1-s}} \mbox{ and } -\frac{1}{q-1}\left(1 + \frac{\zeta_q(s+1)}{\zeta_q(2-s)}q^{1-s}\right) = \frac{1}{1-q^s}.$$
Finally, rewriting
$$\frac{a_{A+1}(\chi)}{q^{A+1}{s}} = \sum_{f\in \cM_{\leq A+1}} \frac{\chi(f)}{|f|^s} -\sum_{f\in \cM_{\leq A}} \frac{\chi(f)}{|f|^s}$$
and
$$\frac{a_{g-A}(\chi)}{q^{(g-A)(1-s)}} = \sum_{f\in \cM_{\leq g-A}} \frac{\chi(f)}{|f|^{1-s}} -\sum_{f\in \cM_{\leq g-A-1}} \frac{\chi(f)}{|f|^{1-s}}  $$
we obtain the even approximate functional equation.

\end{proof}

\subsection{Principal and Dual Terms}

We see that since $\fH(3g) = \fH_o(3g)$, we always have the functional equation
$$\cL(q^{-s}, \chi) =  \sum_{f\in \mathcal{M}_{\leq A}} \frac{\chi(f)}{|f|^s} + \omega(\chi) q^{(\frac12-s)3g} \sum_{f\in \mathcal{M}_{\leq 3g-A-1}} \frac{\overline{\chi}(f)}{|f|^{1-s}}.$$
Further, it will be convenient for computations if we split at an integer that is divisible by $3$. 
Therefore, we define the principal sum
\begin{align*}
    \mathcal{P}_s(3g,3A) := \sum_{\chi\in \fH(3g)} \sum_{f\in \cM_{\leq 3A}} \frac{\chi(f)}{|f|^s} 
\end{align*}
and the dual sum
\begin{align*}
    \mathcal{D}_s(3g,3A) = q^{(\frac{1}{2}-s)3g} \sum_{\chi \in \fH(3g)} \omega(\chi) \sum_{f\in \cM_{\leq 3g-3A-1}} \frac{\overline{\chi}(f)}{|f|^{1-s}},
\end{align*}
so that, for any $0<A<g$, we have
\begin{align}\label{AFEApplied}
    \frac{1}{|\fH(3g)|} \sum_{\chi \in \fH(3g)} \mathcal{L}(q^{-s},\chi) = \frac{\mathcal{P}_s(3g,3A) + \mathcal{D}_s(3g,3A)}{|\fH(3g)|}
\end{align}

\subsection{Perron's Formula}

\begin{prop} \label{perron} If the generating series $\mathcal{A}(u) = \sum_{f \in \cM} a(f) u^{\deg{f}}$ is absolutely convergent in $|u| \leq r < 1$, then
\begin{align*}
\sum_{f \in \cM_n} a(f) = \frac{1}{2 \pi i} \oint_{|u|=r} \frac{\mathcal{A}(u)}{u^{n+1}} \; du
\end{align*}
and
\begin{align*}
\sum_{f \in \cM_{\leq n}} a(f) = \frac{1}{2 \pi i} \oint_{|u|=r} \frac{\mathcal{A}(u)}{u^{n+1}(1-u)} \; du.
\end{align*}
\end{prop}

\section{The Principal Sum} \label{section-principal}

\subsection{Contributions of the Principal Sum} 

This section is devoted to proving contribution of the principal sum.

\begin{prop}\label{PrincComp}
Let $\epsilon>0$. For $0 < s < 1$, $s\not=\frac{1}{3}$, we have
$$ \frac{\mathcal{P}_s(3g,3A)}{|\fH(3g)|} = M_q(s)+ C_q \frac{q^{(1-3s)A}}{1-q^{3s-1}} 
+ O_\epsilon\Big(q^{(\epsilon-3s)A} + q^{ A - (1-\epsilon)(3g+1)} + q^{ (1-s+\epsilon)3A- \frac{1}{2}(3g+1)}\Big) $$
where $M_q(s)$ and $C_q $ are as an in Theorem \ref{MainThm}. If $s=\frac{1}{3}$, we have
$$\frac{\mathcal{P}_\frac{1}{3}(3g,3A)}{|\fH(3g)|} = C_q  \left(A+1 + \sum_P \frac{\deg(P)}{|P|^2+|P|-1}\right)  + O_\epsilon\Big(q^{(\epsilon-1)A} + q^{ A - (1- \epsilon)(3g+1)} + q^{ (2+\epsilon)A - \frac{1}{2}(3g+1)}\Big)$$
    
\end{prop}

\subsection{Error Term}

Since we are assuming $q\equiv 1 \mod{6}$, we have by cubic reciprocity (Theorem 3.5 of \cite{Rosen}) that 
$$\chi_{F}(f) = \left(\frac{f}{F}\right)_3 = \left(\frac{F}{f}\right)_3 = \chi_{f}(F).$$
Therefore, since $\fH(3g) = \{\chi_F : F\in \mathcal{H}(3g+1)\}$ we may rewrite 
$$\mathcal{P}_s(3g,3A) = \sum_{f\in \mathcal{M}_{\leq 3A}} \frac{1}{|f|^s} \sum_{F\in \cH(3g+1)} \chi_{f}(F) .$$

We now wish to compute the innermost sum. To do this, we consider the generating series
$$\fP(u;f) := \sum_{d=0}^{\infty} \sum_{F\in \cH(d)} \chi_f(F) u^d.$$
This series converges for $|u|<q^{-1}$.

\begin{lem} \label{lemma-Lindelof}
If $f$ is not a cube, then 
$$\fP(u;f) = \frac{\cL(u,\chi_f)}{\cL(u^2,\overline{\chi}_f)}.$$ 
In particular, $\fP(u;f)$ can be analytically extended to the region $|u|<q^{-1/4-\epsilon}$. Moreover, if $\Gamma_1 = \{u : |u|=q^{-1/2} \}$,  then for any $\epsilon > 0$, 
$$\max_{u\in \Gamma_\epsilon} |\fP(u,f)| \ll q^{\epsilon \deg f}$$
\end{lem}

\begin{proof}

Since we are summing over all square-free polynomials, we get the Euler product
\begin{align*}
    \fP(u;f) & = \prod_P \left(1+\chi_f(P)u^{\deg P}\right) = \prod_P \left(\frac{1-\chi^2_f(P)u^{2\deg P}}{1-\chi_f(P)u^{\deg P}}\right)
    \\&  = \frac{\cL(u,\chi_f)}{\cL(u^2,\overline{\chi}_f)} 
\end{align*}
where we have used the fact that $\chi_f$ is a cubic character and thus $\chi_f^2 = \overline{\chi_f}$. Now, since $\chi_f$ is non-trivial ($f$ is not a cube), $\cL(u, \chi_f)$ and $\cL(u^2, \overline{\chi}_f)$ are analytic for all $u$, and  $\fP(u;f)$ is analytic for $|u| < q^{-1/4 - \epsilon}$, since this region does not contain the zeroes of $\cL(u^2, \overline{\chi}_f)$.
Furthermore, since again $\chi_f$ is not trivial, 
for $|u| = q^{-1/2}$, we have bounds
\begin{align*}
\cL(u, \chi_f) &\ll_\epsilon q^{\epsilon \deg{f}} \\
\cL(u^2, \overline{\chi_f}) &\gg_\epsilon q^{-\epsilon \deg{f}} 
\end{align*}
for any $\epsilon > 0$. The first bound is the Lindel\"of hypothesis \cite[Theorem 5.1]{BCDGL} and the second bound is proven in \cite[Lemma 2.6]{DFL}.

\kommentar{Now, we can complete the $L$-functions and then write them in terms of the Frobenii to get
\begin{align*}
    \fP(u;f) &  = \frac{\cL(u,\chi_f)}{\cL(u^2,\overline{\chi}_f)} = \frac{(1-u)^{1-\delta_{\chi_f}} \cL_C(u,\chi_f) }{ (1-u^2)^{1-\delta_{\chi_f}} \cL_C(u^2,\overline{\chi}_f) }\\
    & = \frac{\det(1-\sqrt{q}u\Theta_{\chi_f})}{ (1+u)^{1-\delta_{\chi_f}} \det(1-\sqrt{q}u^2\overline{\Theta}_{\chi_f})}
\end{align*}

Thus we see that $\fP(u;f)$ has poles only when $|u|=q^{-1/4}$ and (possibly) at $u=-1$. This shows that $\fP(u;f)$ can be analytically extended to the region $|u|<q^{-1/4}$. 

Recall that $\Theta_{\chi_f}$ is a unitary matrix with dimension $N$ at most $\deg(f)$. Now, if $u\in \Gamma_1$, then we get
\begin{align*}
    |\cL(u,\chi_f)| & = |(1-u)^{1-\delta_{\chi_f}}\det(1-\sqrt{q}u\Theta_{\chi_f})| \\
    & \leq (1+q^{-1/2-\epsilon})^{1-\delta_{\chi_f}} \prod_{j=1}^N (1+q^{-\epsilon}) \ll (1+q^{-\epsilon})^{\deg(f)}
\end{align*}
Finally, we find that for $u\in\Gamma_1$ then by the reverse triangle inequality
\begin{align*}
    |\cL(u^2,\overline{\chi_f})| & = |(1-u^2)^{1-\delta_{\chi_f}} \det(1-\sqrt{q}u^2\overline{\Theta}_{\chi_f})|\\
    & \geq |(1-q^{-1-2\epsilon})^{1-\delta_{\chi_f}} \prod_{j=1}^N (1-q^{-1/2-2\epsilon})| \gg 1
\end{align*}

Combining the two inequalities above, we obtain
$$\max_{u\in\Gamma_1}|\fP(u;f)| \ll_\epsilon (1+q^{-\epsilon})^{\deg(f)}$$}

\end{proof}

\kommentar{\begin{lem}

If $f$ is not a cube, then for any $\epsilon>0$,
$$\sum_{F\in \cH(d)}\chi_f(F) \ll_\epsilon q^{\epsilon \deg(f)} q^{\frac{1}{2} d}.$$

\end{lem}

\begin{proof}

With $\Gamma_1$ as in the previous lemma, we get that $\frac{\fP(u;f)}{u^{d+1}}$ is meromorphic in the region bound by $\Gamma_1$ with a pole only at $u=0$. Hence,
$$\frac{1}{2\pi i} \oint_{\Gamma_1} \frac{\fP(u;f)}{u^{d+1}} du = \Res_{u=0}\left( \frac{\fP(u;f)}{u^{d+1}} \right) = \sum_{F\in \cH(d)} \chi_f(F).$$
Now, also by the previous Lemma, we have that
$$ \frac{1}{2\pi i} \oint_{\Gamma_1} \frac{\fP(u;f)}{u^{d+1}} du \ll_\epsilon \max_{u\in \Gamma_1} \left| \frac{\fP(u;f)}{u^{d+1}}  \right| \ll_\epsilon q^{\epsilon \deg(f)} q^{\frac{1}{2} d} $$
which completes the proof.

\end{proof}}

\begin{lem}\label{PrincErrTermLem}
For any $\epsilon>0$, we have
$$\frac{1}{|\mathcal{H}(d)|}\sum_{\substack{ f\in \cM_{\leq A} \\  f\not=\mbox{\mancube} }} \frac{1}{|f|^s} \sum_{F\in \cH(d)} \chi_f(F) \ll_\epsilon q^{ (1-s + \epsilon)A - \frac{1}{2}d }.$$
\end{lem}

\begin{proof}
We first show that when $f$ is not a cube,
\begin{align} \label{not-a-cube}
\sum_{F\in \cH(d)}\chi_f(F) \ll_\epsilon q^{\epsilon \deg f} q^{\frac{1}{2} d}.
\end{align}
With $\Gamma_1$ as in the previous lemma, we get that $\frac{\fP(u;f)}{u^{d+1}}$ is meromorphic in the region bounded by $\Gamma_1$ with a pole only at $u=0$. Hence,
$$\frac{1}{2\pi i} \oint_{\Gamma_1} \frac{\fP(u;f)}{u^{d+1}} du = \Res_{u=0}\left( \frac{\fP(u;f)}{u^{d+1}} \right) = \sum_{F\in \cH(d)} \chi_f(F),$$
and 
$$ \frac{1}{2\pi i} \oint_{\Gamma_1} \frac{\fP(u;f)}{u^{d+1}} du \ll_\epsilon \max_{u\in \Gamma_1} \left| \frac{\fP(u;f)}{u^{d+1}}  \right| \ll_\epsilon q^{\epsilon \deg f} q^{\frac{1}{2} d} $$
which completes the proof of \eqref{not-a-cube}.
Applying this result, we then get that
\begin{align*}
    \sum_{\substack{ f\in \cM_{\leq A} \\  f\not=\mbox{\mancube} }} \frac{1}{|f|^s} \sum_{F\in \cH(d)} \chi_f(F) & \ll_\epsilon  q^{\frac{1}{2} d}  \sum_{\substack{ f\in \cM_{\leq A} \\  f\not=\mbox{\mancube} }} |f|^{(\epsilon - s)} \\   & \ll  q^{\frac{1}{2} d} \sum_{n\leq A} q^n q^{n(\epsilon - s)}  \ll q^{\frac{1}{2} d} q^{(1-s+\epsilon)A}
\end{align*}
and the result follows from the fact that $|\mathcal{H}(d)| = \frac{q^d}{\zeta_q(2)}$.

\end{proof}

\subsection{Main Term}

In the case that $f = h^3$ is a perfect cube, then
$$\chi_f(F) = \left(\frac{F}{f}\right)_3 = \left(\frac{F}{h}\right)_3^3 = \begin{cases} 1 & (F,f)=1 \\ 0  & (F,f)\not=1\end{cases}.$$
Hence, in the case that $f$ is a perfect cube, we get that
$$\sum_{F\in \cH(d)} \chi_f(F) = |\{\cH(d,f)\}|$$
where
$$\cH(d,f) = \{F\in \cH(d) : (F,f)=1\}.$$

Therefore, we consider the generating series
\begin{align*}
    \mathcal{Q}(u;f)&  = \sum_{d=0}^{\infty} |\cH(d,f)|u^d = \sum_{(F,f)=1} \mu^2(F) u^{\deg{F}}\\
    & = \prod_{P\nmid f} (1+u^{\deg P}) = \prod_{P|f} (1+u^{\deg P})^{-1} \prod_P \left(\frac{1-u^{2\deg P}}{1-u^{\deg P}}\right) \\
    & = \prod_{P|f}(1+u^{\deg P})^{-1} \frac{\cZ_q(u)}{\cZ_q(u^2)} 
     = \prod_{P|f}(1+u^{\deg P})^{-1} \frac{1-qu^2}{1-qu}
\end{align*}
We see that $\mathcal{Q}(u;f)$ can be meromorphically extended to the region $|u|<1$ with a simple pole at $u=q^{-1}$. Using Proposition \ref{perron}, we write
\begin{align*}
 |\cH(d;f)| =  \frac{1}{2\pi i} \oint_{|u|=q^{-2}} \frac{\mathcal{Q}(u;f)}{u^{d+1}} du .
\end{align*}
Moving the contour to
$$\Gamma_2 = \left\{u : |u| = q^{-\epsilon}\right\},$$
we get that
\begin{align} \label{after-Perron}
 |\cH(d;f)| = \frac{1}{2\pi i} \oint_{\Gamma_2} \frac{\mathcal{Q}(u;f)}{u^{d+1}} du - \Res_{u=q^{-1}}\left(\frac{\mathcal{Q}(u;f)}{u^{d+1}}\right).
\end{align}
For $d\geq 1$, we compute
\begin{align*}
    \Res_{u=q^{-1}}\left(\frac{\mathcal{Q}(u;f)}{u^{d+1}}\right) & = \lim_{u\to q^{-1}} \left(\frac{(u-q^{-1})}{u^{d+1}}\prod_{P|f}(1+ u^{\deg P})^{-1} \frac{1-qu^2}{1-qu}\right) \\
    & = - \prod_{P|f}\left(1+\frac{1}{|P|}\right)^{-1} \left(q^d-q^{d-1}\right).
\end{align*}
For $u\in \Gamma_2 \iff |u|=q^{-\epsilon}$, we use
\begin{align*}
\prod_{P|f} (1+u^{\deg P})^{-1} \leq \prod_{P|f} (1- |u|^{\deg P})^{-1}
\ll_\epsilon \prod_{P|f} \left( 1 + |u|^{\deg{P}} \right) \leq  \sum_{d \mid f} |u|^{\deg{f}} \ll_\epsilon |f|^\epsilon,
\end{align*}
bounding the number of divisors of $f$ by $|f|^\epsilon$. Then, for $u\in \Gamma_2$, we have
\begin{align*}
    |\mathcal{Q}(u;f)| & =\left| \prod_{P|f} (1+u^{\deg P})^{-1} \frac{1-qu^2}{1-qu}\right|  \ll_\epsilon  |f|^\epsilon,
\end{align*}
and replacing in \eqref{after-Perron}, we conclude that 
\begin{align} \label{H(d;f)}
| \mathcal{H}(d;f) | = \prod_{P|f}\left(1+\frac{1}{|P|}\right)^{-1}|\cH(d)| + O_\epsilon \left( |f|^\epsilon q^{\epsilon d} \right).
\end{align}

\begin{lem}\label{PrincMainTermLem}
Let $\epsilon >0$, and $\epsilon < s < 1$.  Then if $s\not = \frac{1}{3}$, we get
\begin{multline*}
    \frac{1}{|\cH(d)|} \sum_{\substack{ f\in \cM_{\leq 3A} \\  f=\mbox{\mancube} }} \frac{1}{|f|^s} \sum_{F\in \cH(d)} \chi_f(F)= M_q(s)+ C_q \frac{q^{(1-3s)A}}{1-q^{3s-1}}
    + O_\epsilon \left( q^{(\epsilon-3s)A} + q^A q^{  (\epsilon-1)d} \right)
\end{multline*}
where $M_q(s)$ and $C_q$ are as in the Theorem \ref{MainThm}. If $s=\frac{1}{3}$, then we get
\begin{multline*}
    \frac{1}{|\cH(d)|} \sum_{\substack{ f\in \cM_{\leq 3A} \\  f=\mbox{\mancube} }} \frac{1}{|f|^s} \sum_{F\in \cH(d)} \chi_f(F)=  C_q \left(A+1 + \sum_P \frac{1}{|P|^2+|P|-1}\right)  
    + O_\epsilon \left( q^{(\epsilon-1)A} + q^A q^{ (\epsilon-1)d} \right).
\end{multline*}
\end{lem}

\begin{proof}

Indeed, if we write $f=h^3$, then we get using \eqref{H(d;f)}
\begin{align} \nonumber
    &\sum_{\substack{ f\in \cM_{\leq 3A} \\  f=\mbox{\mancube} }} \frac{1}{|f|^s} \sum_{F\in \cH(d)} \chi_f(F) = \sum_{h\in \cM_{\leq A}} \frac{|\cH(d,h)|}{|h|^{3s}}  \\
    \label{sum-cubes}
    = &  \sum_{h\in \cM_{\leq A}} \frac{1}{|h|^{3s}}\left[ \prod_{P|h}\left(1+ \frac{1}{|P|}\right)^{-1}|\cH(d)| + O_\epsilon\left( q^{\epsilon \deg h} q^{\epsilon d} \right)\right].
\end{align}
First, we compute the error term and find
\begin{align} \label{sum-cubes-ET}
    \sum_{h\in \cM_{\leq A}} \frac{q^{\epsilon \deg h}q^{\epsilon d}}{|h|^{3s}} = q^{\epsilon d} \sum_{0 \leq m\leq A} \left(\frac{q^\epsilon}{q^{3s}}\right)^m \ll q^{\epsilon d},
\end{align}
which gives the second error term upon dividing by $|\cH(d)|$.
For the main term, we define the generating series
$$\mathcal{G}_s(v):=\sum_{h\in \cM} \prod_{P|h}\left(1+\frac{1}{|P|}\right)^{-1} \frac{v^{\deg h}}{|h|^{3s}}.$$
Expanding it as an Euler product, we see that
\begin{align}\label{PrinceGenSer}
    \mathcal{G}_s(v) & = \prod_P \left( 1 + \left(1+\frac{1}{|P|}\right)^{-1}\sum_{k=1}^{\infty} \left(\frac{v}{q^{3s}}\right)^{k\deg P} \right) \\
    & = \prod_P \left( \frac{1-\left(\frac{v}{q^{3s}}\right)^{\deg P} + \left(1-\frac{1}{|P|+1}\right)\left(\frac{v}{q^{3s}}\right)^{\deg P}}{ 1-\left(\frac{v}{q^{3s}}\right)^{\deg P} }\right) \nonumber \\
    & = \cZ_q\left(\frac{v}{q^{3s}}\right) \prod_P \left(1-\frac{v^{\deg P}}{|P|^{3s}(|P|+1)}\right), \nonumber 
\end{align}
and  $\mathcal{G}_s(v)$ can be meromorphically continued to the region $|v|\leq q^{3s-\epsilon}$ with a simple pole when $v=q^{3s-1}$. 
Then, using Proposition \ref{perron},
\begin{align*}
\sum_{h\in \cM_{\leq A}} \frac{1}{|h|^{3s}}  \prod_{P|h}\left(1+\frac{1}{|P|}\right)^{-1} = \frac{1}{2 \pi i} \oint_{|u|=q^{-2}} \frac{\mathcal{G}_s(v)}{1-v} \frac{dv}{v^{A+1}},
\end{align*}
and moving the contour to $\Gamma_3 = \{v : |v|=q^{3s-\epsilon}\}$, where 
 $s>\epsilon$ and $s\not=\frac13$, we get that
\begin{align} \nonumber
\sum_{h\in \cM_{\leq A}} \frac{1}{|h|^{3s}}  \prod_{P|h}\left(1+\frac{1}{|P|}\right)^{-1}   &=  \frac{1}{2\pi i} \oint_{\Gamma_3} \frac{\mathcal{G}_s(v)}{1-v} \frac{dv}{v^{A+1}}  \\ 
\label{2-residues} & \quad  -\Res_{v=q^{3s-1}} \left(\frac{\mathcal{G}_s(v)}{(1-v)v^{A+1}}\right)  
- \Res_{v=1} \left(\frac{\mathcal{G}_s(v)}{(1-v)v^{A+1}}\right).
\end{align}
Furthermore, using \eqref{sum-cubes} and \eqref{sum-cubes-ET}, \begin{align*}
 \sum_{h\in \cM_{\leq A}} \frac{1}{|h|^{3s}}  \prod_{P|h}\left(1+\frac{1}{|P|}\right)^{-1} =
\frac{1}{|\cH(d)|} \sum_{\substack{ f\in \cM_{\leq 3A} \\  f=\mbox{\mancube} }} \frac{1}{|f|^s} \sum_{F\in \cH(d)} \chi_f(F) + O_\epsilon\left( q^A q^{(\epsilon -1)d} \right)
\end{align*}
and 
$$\left|\frac{1}{2\pi i} \oint_{\Gamma_3} \frac{\mathcal{G}_s(v)}{1-v} \frac{dv}{v^{A+1}}\right| \ll \frac{\max_{v\in \Gamma_3}\left|\frac{\mathcal{G}_s(v)}{1-v} \right|}{q^{(3s-\epsilon)A}} \ll_\epsilon q^{(\epsilon-3s)A} .$$
So, we have obtained the error terms of the lemma, and 
it remains to compute the residues in \eqref{2-residues}.
By \eqref{PrinceGenSer}, we see that
\begin{align*}
     \Res_{v=q^{3s-1}} \left(\frac{\mathcal{G}_s(v)}{(1-v)v^{A+1}}\right) & = -\frac{q^{(1-3s)A}}{1-q^{3s-1}}\prod_P \left(1-\frac{1}{|P|(|P|+1)}\right) 
\end{align*}
while
\begin{align*}
    \Res_{v=1} \left(\frac{\mathcal{G}_s(v)}{(1-v)v^{A+1}}\right) & = - \zeta_q(3s)\prod_P \left(1-\frac{1}{|P|^{3s}(|P|+1)}\right) = -M(s)
\end{align*}
which converges because $s>\epsilon$, $s\not=\frac{1}{3}$. This completes the proof for $s\not=\frac{1}{3}$.

Now, if $s=\frac13$, working as above, we get 
\begin{align} \nonumber
\sum_{h\in \cM_{\leq A}} \frac{1}{|h|^{3s}}  \prod_{P|h}\left(1+\frac{1}{|P|}\right)^{-1}   &=  
- \Res_{v=1} \left(\frac{\mathcal{G}_s(v)}{(1-v)v^{A+1}}\right) + O \left(  q^{(\epsilon-1)A} + q^A q^{ (\epsilon-1)d} \right),
\end{align}
where now have a double pole at $v=1$, and it remains to compute the residue at $v=1$.
Denoting
$$\mathcal{K}(v) := \prod_P \left(1-\frac{v^{\deg P}}{|P|(|P|+1)}\right)$$
we obtain that
\begin{align*}
\Res_{v=1}\left(\frac{\mathcal{G}_{\frac{1}{3}}(v)}{(1-v)v^{A+1}}\right) & = \lim_{v\to 1} \frac{d}{dv} \frac{\mathcal{K}(v)}{v^{A+1}} = -\mathcal{K}(1)(A +1) + \mathcal{K}'(1).
\end{align*}
The result now follows from the fact that $C_q = \mathcal{K}(1)$, and 
$$\frac{\mathcal{K}'(v)}{\mathcal{K}(v)} = \frac{d}{dv}\log(\mathcal{K}(v)) = -\sum_P \frac{v^{\deg{P}-1} \, \deg{P} }{|P|(|P|+1)-v^{\deg P}},$$
which completes the proof of Lemma \ref{PrincMainTermLem}.
\end{proof}

\begin{proof}[Proof of  Proposition \ref{PrincComp}.]
We can now combine Lemmas \ref{PrincErrTermLem} (with $d=3g+1$ and replacing $A$ by $3A$) and \ref{PrincMainTermLem} (with $d=3g+1$) to prove Proposition \ref{PrincComp}.
\end{proof}

\section{The Dual Term}\label{section-dual-sum} 

\subsection{Contributions of the Dual  Sum} 
This section is devoted to proving contribution of the dual terms.

\begin{prop}\label{DualComp}
Let $\epsilon>0$ and $0<s<1$, $s\not=\frac{1}{3}$. Then, 
\begin{align*}
    \frac{\mathcal{D}_s(3g,3A)}{|\fH(3g)|} = -C_q \frac{q^{(1-3s)A}}{1-q^{3s-1}} + O_\epsilon\left( \frac{q^{3(1-s)A+2-2s}}{q^{(2-3\epsilon)g}} + \frac{q^{(\frac{3}{4} + \epsilon)g}}{q^{3(s+\frac{1}{4} + \epsilon)A}} +  q^{(\frac{1}{3}-s)(3g+1)} E_s(3g,3A)\right)
\end{align*}
where $C_q$ is as in Theorem \ref{MainThm} and
$$E_s(3g,3A) = \begin{cases}  1 & s<2/3 \\(g-A)^2 & s=2/3 \\ (g-A)q^{(3s-2)(g-A)} & s>2/3. \end{cases}$$
If $s=\frac{1}{3}$, then
\begin{align*}
    \frac{\mathcal{D}_{\frac{1}{3}}(3g,3A)}{|\fH(3g)|} =  C_q \left(g-A - \frac{1}{3}\frac{q+2}{q-1} +\sum_{P} \frac{\deg(P)}{|P|^3+2|P|^2-1}\right) + O_\epsilon\left(\frac{q^{(2+\epsilon)A}}{q^{(2-\epsilon)g}} + \frac{q^{(\frac{3}{4}+\epsilon)g}}{q^{(\frac{7}{4}+\epsilon)A}} +  1\right).
\end{align*}
\end{prop}

\begin{rem}
    As in Theorem \ref{MainThm}, when $q$ is fixed and $g-A\to \infty$, we could write the second result as 
    $$\frac{\mathcal{D}_{\frac{1}{3}}(3g,3A)}{|\fH(3g)|} = C_q (g-A) \left( 1 + o(1) \right).$$
    We keep it as is to give evidence of an explicit constant appearing in the shape of a prime sum. 
\end{rem}

\subsection{Applying Results of \cite{DFL}}

Let $\chi_F$ be any cubic character as defined by \eqref{by-mult}. This is then a character modulo $F$, but not necessarily primitive. We define the generalized cubic Gauss sum
\begin{align} \label{generalized-gauss-sum}
G(V,F):= \sum_{a \bmod{F}} \chi_{F}(a) e_q\left(\frac{aV}{F}\right).\end{align}
It is important to notice that if $\chi_F$ has conductor $F'$ with $\deg(F')<\deg(F)$ then $G(1,F)\not=G(\chi_F)$. Conversely, if $F$ is, say square-free, then we do have that $G(1,F) = G(\chi_F)$.

\begin{lem}{\cite[Corollary 2.3, Equation (22)]{DFL}} \label{DFLLem2}
For any $(f,F)=1$ and $V$, we have
$$\overline{\chi}_F(f) G(V,F) = G(fV,F).$$
\end{lem}

Applying Lemmas \ref{DFLLem1} and \ref{DFLLem2}, with the observation that $\chi_f(F) =0$ if $(F,f)=1$, we get that
\begin{align*}
        \mathcal{D}_s(3g,3A) & = q^{(\frac{1}{2}-s)3g}\sum_{f\in \cM_{\leq 3g-3A-1}}\frac{1}{|f|^{1-s}} \sum_{\substack{F\in \cH(3g+1) \\ (F,f)=1}} \overline{\chi}_F(f) \omega(\chi_F)  \\
        & = \frac{\overline{\epsilon(\chi_3)}}{q^{3gs+\frac{1}{2}}} \sum_{f\in\cM_{\leq  3g-3A-1}} \frac{1}{|f|^{1-s}} \sum_{\substack{F\in\cH(3g+1) \\ (F,f)=1}} G(f,F).
\end{align*}

Finally, we state another result that can be found in Lemma 2.12 of \cite{DFL}.

\begin{lem}\label{DFLLem3} If $(F_1,F_2)=1$, then
    \begin{align*}
        G(V,F_1F_2) &= \chi^2_{F_1}(F_2) G(V,F_1)G(V,F_2) \\ & = G(VF_2,F_1)G(V,F_2).
    \end{align*}
    Moreover, if there exists a prime $P$ such that $P^2|F$ and $P\nmid V$, then
    $$G(V,F)=0.$$
\end{lem}

The second part of Lemma \ref{DFLLem3} implies that if $F\in\cM_{3g+1}$ is not square-free and $(F,f)=1$ then $G(f,F)=0$. Therefore, in the formula of $\cD_s(3g,3A)$ above, we may remove the condition that $F$ is square-free, and write
\begin{align}\label{DualGaussSum}
    \mathcal{D}_s(3g,3A) = \frac{\overline{\epsilon(\chi_3)}}{q^{3gs+\frac{1}{2}}} \sum_{f\in\cM_{\leq 3g-3A-1}} \frac{1}{|f|^{1-s}} \sum_{\substack{F\in\cM_{3g+1} \\ (F,f)=1}} G(f,F).
\end{align}


\begin{prop}[Proposition 3.1 of \cite{DFL}]\label{DFLProp}
Let $f=f_1f_2^2f_3^3$ with $f_1$ and $f_2$ square-free and coprime. We have, for any $\epsilon>0$ 
\begin{multline*}
    \sum_{\substack{F\in \cM_d \\ (F,f)=1}} G(f,F) = \delta_{f_2=1}\rho(d;f) \frac{\overline{G(1,f_1)}}{|f_1|^{2/3}} \frac{q^{4d/3}}{\zeta_q(2)} \prod_{P|f} \left(1+\frac{1}{|P|}\right)^{-1}  + O_\epsilon\left( \delta_{f_2=1} \frac{q^{(\frac{1}{3}+\epsilon)d}}{|f_1|^{\frac{1}{6}}} + q^{d}|f|^{\frac{1}{4} + {\epsilon}} \right)
\end{multline*}
where
$$\rho(d;f) = \begin{cases} 1 & d+\deg(f)\equiv 0 \mod{3} \\ \frac{\tau(\chi_3)}{q^{1/3}} & d+\deg(f) \equiv 1 \mod{3} \\ 0 & d+\deg(f)\equiv 2 \mod{3} \end{cases}.$$
\end{prop}

\begin{proof} In \cite[Proposition 3.1]{DFL}, the second error term is of the form
$$\frac{1}{2 \pi i} \int_{|u| = q^{-\sigma}} \frac{\widetilde{\psi}(f, u)}{u^d} \frac{du}{u}$$
for any $\frac23 < \sigma < \frac43$. In that region, we have the convexity bound 
$$
\widetilde{\psi}(f, u) \ll |f|^{\frac12 \left(\frac32 - \sigma \right) + \epsilon}
$$
from \cite[Proposition 3.11]{DFL}. Taking $\sigma =1$, we get the result.
\end{proof}

Using the above proposition in \eqref{DualGaussSum}, we write 
$$\mathcal{D}_s(3g,3A) = MT_s(3g,3A) + O_\epsilon(ET_s(3g,3A))$$
where
$$MT_s(3g,3A)  := \frac{\overline{\epsilon(\chi_3)}q^{(\frac{4}{3}-s)3g+\frac{5}{6}}}{\zeta_q(2)} \sum_{f\in\cM_{\leq 3g-3A-1}} \frac{\delta_{f_2=1} \rho(1;f) \overline{G(1,f_1)}}{|f|^{1-s}|f_1|^{2/3}}   \prod_{P|f} \left(1+\frac{1}{|P|}\right)^{-1}$$
and 
$$ET_s(3g,3A) := \frac{1}{q^{3gs+\frac{1}{2}}} \sum_{f\in\cM_{\leq 3g-3A-1}} \frac{1}{|f|^{1-s}}\left( \delta_{f_2=1} \frac{q^{(\frac{1}{3}+\epsilon)(3g+1)}}{|f_1|^{\frac{1}{6}}} + q^{3g+1}|f|^{\frac{1}{4} + \varepsilon} \right),$$
where we use the fact that $|\epsilon(\chi_3)|=1$.
\subsection{Bounding the Error Term}

\begin{lem}\label{D_ETBound}
    For any $\epsilon>0$ and $s\geq 0$, we get that
    $$ET_s(3g,3A) \ll
     q^{(\frac{15}{4} + \epsilon)g - 3(s+\frac{1}{4}+\epsilon)A }$$
\end{lem}

\begin{proof}
For the second sum of $ET_s(3g,3A)$, we have 
\begin{align*}
\frac{q^{3g+1}}{q^{3gs+\frac{1}{2}}}\sum_{f\in \cM_{\leq 3g-3A-1}} \frac{|f|^{\frac{1}{4}+\epsilon}}{|f|^{1-s}} = q^{3g(1-s)+\frac{1}{2}} \sum_{n\leq 3g-3A-1}  q^{(s+\frac{1}{4}+\epsilon)n} \ll q^{(\frac{15}{4} + \epsilon)g - 3(s+\frac{1}{4}+\epsilon)A .}
\end{align*}
 For the first sum of $ET_s(3g,3A)$, if $f_2=1$, then we can write $f=f_1f_3^3$ where $f_1$ is square-free and $f_3$ is anything. Hence, setting $B=3g-3A-1$, we get
\begin{align*}
\frac{1}{q^{3gs+\frac{1}{2} }}\sum_{f\in\cM_{\leq B}} \frac{\delta_{f_2=1}q^{(\frac{1}{3}+\epsilon)(3g+1)}}{|f|^{1-s}|f_1|^{\frac{1}{6}}} & = \frac{q^{\frac{5}{6}+\epsilon}}{ q^{(s-\frac{1}{3}-\epsilon)3g}}\sum_{n_1+3n_3\leq B} \sum_{\substack{ f_1\in \cH(n_1) \\ f_3\in \cM_{n_3}}} \frac{1}{q^{(7/6-s)n_1 + 3(1-s)n_3}} \\
& \ll \frac{q^{\frac{5}{6}+\epsilon}}{ q^{(s-\frac{1}{3}-\epsilon)3g}} \sum_{n_1\leq B} q^{(s-1/6)n_1} \sum_{n_3\leq (B-n_1)/3} q^{(3s-2)n_3} \\
& \ll \frac{q^{\frac{5}{6}+\epsilon}}{ q^{(s-\frac{1}{3}-\epsilon)3g}} \sum_{n_1\leq B} q^{(s-1/6)n_1}  \begin{cases} 1 & s<2/3 \\ \frac{B-n_1}{3} & s=2/3 \\ q^{(s-\frac{2}{3})(B-n_1)}  & s>2/3 , \end{cases}
\end{align*}
and working case by case with trivial bounds, it is straightforward to see that this error term is also bounded by $q^{(\frac{15}{4} + \epsilon)g - 3(s+\frac{1}{4}+\epsilon)A }.$




\end{proof}

\subsection{Extending Proposition \ref{DFLProp}}

Since there is a factor of $\delta_{f_2=1}$ in the sum over $f$ in $MT_s(3g,3A)$, we can write $f = f_1f_3^3$ where $f_1\in\cH(n_1)$ and $f_3\in \cM_{n_3}$. Moreover, since $G(1,f)=0$ unless $f$ is square-free, we may extend the sum over $\cH(n_1)$ to the sum over $\cM_{n_1}$. Hence, we can rewrite
\begin{align} \nonumber
    MT_s(3g,3A) &= \frac{\overline{\epsilon(\chi_3)}q^{(\frac{4}{3}-s)3g+\frac{5}{6}}}{\zeta_q(2)} \sum_{n_1+3n_3\leq 3g-3A-1} \frac{\rho(n_1+1;1)}{q^{(1-s)(n_1 + 3n_3)}}  \\ \label{MTs}
    & \hspace{0.4cm} \times \sum_{\substack{f_1\in \cM_{n_1} \\ f_3\in \cM_{n_3}}} \frac{\overline{G(1,f_1)}}{|f_1|^{2/3}}\prod_{P|f_1f_3} \left(1+\frac{1}{|P|}\right)^{-1}
\end{align}
We now extend Proposition \ref{DFLProp} to get an estimate when $G(1, F)$ is multiplied by an Euler product.

\begin{prop}\label{DFLPropExt}
For any values $a_P$ such that $|a_P| \leq \frac{1}{|P|}$, we have  for any $H\in\mathcal{M}$,
$$\sum_{F\in \cM_n} \frac{G(1,F)}{|F|^{2/3}} \prod_{\substack{P|F \\ P\nmid H}} (1-a_P) = \frac{\rho(n;1)q^{2n/3}}{\zeta_q(2)} \prod_{P\nmid H} \left(1-\frac{a_P}{|P|+1}\right) + O\left( q^{ n/3} \right) .$$
\end{prop}


\begin{proof}

We first expand the Euler product as
$$\prod_{\substack{P|F \\ P\nmid H}} (1-a_P) = \sum_{\substack{D|F \\ (D,H)=1}} \mu(D) a_D, \quad \quad \mbox{ where } \quad \quad a_D := \prod_{P|D} a_P.$$
Notice that by hypothesis on $a_P$,  have  $|a_D| \leq \frac{1}{|D|}$ which we will use often.
Expanding the Euler product like this, we obtain
\begin{align}\nonumber
    \sum_{F\in \cM_n} \frac{G(1,F)}{|F|^{2/3}} \prod_{\substack{P|F \\ P\nmid H}} (1-a_P) & = \frac{1}{q^{2n/3}} \sum_{F\in \cM_n} G(1,F) \sum_{\substack{D|F \\ (D,H)=1}} \mu(D)a_D \\ \label{sum-over-D}
    & = \frac{1}{q^{2n/3}} \sum_{\substack{D\in \cM_{\leq n} \\ (D,H)=1}}   \mu(D) a_D \sum_{F\in \cM_{n-\deg(D)}} G(1,FD).
\end{align}
The idea now is to use the work of \cite{DFL} (specifically Proposition \ref{DFLProp} which is \cite[Proposition 3.1]{DFL}) to evaluate this innermost sum.

Now, Lemmas \ref{DFLLem2} and \ref{DFLLem3} shows that
$$G(1,FD) = \begin{cases} G(D,F)G(1,D) & (F,D)=1 \\ 0 & (F,D)\not=1 \end{cases}.$$
Applying Proposition \ref{DFLProp} we get (since $D$ is always square-free) that
\begin{align*}
    \sum_{F \in \cM_{n-\deg(D)}} G(1,FD)  = & G(1,D)\sum_{\substack{F \in \cM_{n-\deg(D)} \\ (F,D)=1 }}G(D,F) \\
    = & G(1,D)\rho(n-\deg D ; D) \frac{\overline{G(1,D)}}{|D|^{2/3}} \frac{q^{4(n-\deg D)/3}}{\zeta_q(2)} \prod_{P|D} \left(1+\frac{1}{|P|}\right)^{-1} \\ & + O_\epsilon\left(G(1,D)\left(  \frac{q^{(\frac{1}{3}+\epsilon)(n-\deg D)}}{|D|^{\frac{1}{6}}} + q^{(n-\deg D)}|D|^{\frac{1}{4}+\epsilon} \right)\right) \\
    & = \frac{\rho(n;1)}{\zeta_q(2)} \frac{q^{4n/3}}{|D|} \prod_{P|D} \left(1+\frac{1}{|P|}\right)^{-1} + O_\epsilon\left(  \frac{q^{(\frac{1}{3}+\epsilon)n}}{|D|^{\epsilon}} + \frac{q^{n}}{|D|^{\frac{1}{4}-\epsilon}} \right).
 \end{align*}
Hence, summing the main term over $D$, we find that
\begin{align*}
    \sum_{\substack{D\in \cM_{\leq n}\\ (D,H)=1}} \frac{\mu(D)a_D}{|D|} \prod_{P|D}\left(1+\frac{1}{|P|}\right)^{-1} & = \sum_{\substack{D\in \cM\\ (D,H)=1}} 
    \frac{\mu(D)a_D}{|D|} \prod_{P|D}\left(1+\frac{1}{|P|}\right)^{-1} + O\left(\frac{1}{q^n}\right)  \\
    & = \prod_{P\nmid H} \left(1 - \frac{a_P}{|P|+1}\right) + O\left(\frac{1}{q^n}\right) ,
\end{align*}
\and summing the error term over $D$, we find that
$$\sum_{\substack{D\in \cM_{\leq n}\\ (D,H)=1}} a_D \left(  \frac{q^{(\frac{1}{3}+\epsilon)n}}{|D|^{\epsilon}} + \frac{q^{n}}{|D|^{\frac{1}{4} - \varepsilon}} \right) \ll q^{n}.$$
Replacing in \eqref{sum-over-D},
this  finishes the proof.
\end{proof}

Applying Proposition \ref{DFLPropExt}, we obtain
\begin{align*}
    \sum_{\substack{f_1\in \cM_{n_1} \\ f_3\in \cM_{n_3}}} \frac{\overline{G(1,f_1)}}{|f_1|^{2/3}} &\prod_{P|f_1f_3} \left(1+\frac{1}{|P|}\right)^{-1} \\
   &=  \sum_{f_3\in \cM_{n_3}} \prod_{P|f_3}\left(1+\frac{1}{|P|}\right)^{-1} \sum_{f_1\in \cM_{n_1}} \frac{\overline{G(1,f_1)}}{|f_1|^{2/3}} \prod_{\substack{P|f_1\\P\nmid f_3}} \left(1-\frac{1}{|P|+1}\right)\\
    &=  \frac{\overline{\rho(n_1;1)}q^{2n_1/3}}{\zeta_q(2)} \prod_P \left(1-\frac{1}{(|P|+1)^2}\right) \sum_{f_3\in \cM_{n_3}} \prod_{P|f_3}\left(1 +\frac{1}{|P|+1}\right)^{-1}   \\
    & \hspace{0.4cm} + O_\epsilon\left(q^{ n_1/3+n_3}\right)
\end{align*}
Now, with the observation that $\rho(m+3n;1)=\rho(m;1)$ and
$$q^{2n_1/3} = \sum_{f_1\in \cM_{n_1}} \frac{1}{|f_1|^{1/3}}$$
we can write
\begin{align*}
    & \sum_{n_1+3n_3\leq 3g-3A-1} \frac{\rho(n_1+1;1)\overline{\rho(n_1;1)}q^{2n_1/3} }{q^{(1-s)(n_1+3n_3)}} \sum_{f_3\in \cM_{n_3}} 
    \prod_{P|f_3}\left(1 +\frac{1}{|P|+1}\right)^{-1}\\
    = & \sum_{n\leq 3g-3A-1} \frac{\rho(n+1;1)\overline{\rho(n;1)}}{q^{(1-s)n}} \sum_{f_1f_3^3 \in \cM_n}  \frac{1}{|f_1|^{1/3}} \prod_{P|f_3}\left(1 +\frac{1}{|P|+1}\right)^{-1}.
\end{align*}
Using the definition of $\rho(d;f)$, we find that 
$$\overline{\epsilon(\chi_3)}q^{5/6} \rho(n+1;1)\overline{\rho(n;1)} = \begin{cases} q & 3|n \\ 0 & 3\nmid n. \end{cases}$$
Replacing the above in \eqref{MTs}, we can write
$$MT_s(3g,3A) = MMT_s(3g,3A)+O(MET_s(3g,3A))$$
where
\begin{align} \nonumber
    MMT_s(3g,3A) &:= \frac{q^{(\frac{4}{3}-s)3g+1}}{\zeta^2_q(2)} \prod_P \left(1 - \frac{1}{(|P|+1)^2}\right) \\   \label{MMTs}
    & \quad \times\sum_{m\leq g-A-1} \frac{1}{q^{(1-s)3m}} 
  \sum_{f_1f_3^3 \in \cM_{3m}}  \frac{1}{|f_1|^{1/3}} \prod_{P|f_3}\left(1 +\frac{1}{|P|+1}\right)^{-1}
\end{align}
and
$$MET_s(3g,3A) := q^{(\frac{4}{3}-s)3g +1 }\sum_{n_1+3n_3\leq 3g-3A-1} \frac{q^{n_1/3+n_3}}{q^{(1-s)(n_1+3n_3)}}.$$

\begin{lem}\label{D_METBound}

For any value of $A$, we get that
$$MET_s(3g,3A)\ll q^{(\frac{4}{3}-s)(3g+1)} E_s(3g,3A)$$
\end{lem}

\begin{proof}

Indeed, we have
\begin{align*}
    MET_s(3g,3A) & = q^{(\frac{4}{3}-s)3g+1} \sum_{n\leq 3g-3A-1} q^{(s-2/3)n} \sum_{n_1+3n_3=n} 1 \\
    & \ll q^{(\frac{4}{3}-s)3g+1}  \begin{cases}  1 & s<2/3 \\(g-A)^2 & s=2/3 \\ (g-A)q^{(3s-2)(g-A)} & s>2/3 \end{cases}
\end{align*}
    
\end{proof}

\subsection{Computing the Main Dual Term}

\begin{lem}\label{MMTlem}
Let $C_q$ is as in Theorem \ref{MainThm}.
For any $\epsilon>0$ and  $s<1-\epsilon$, if $s\not=\frac{1}{3}$, we have
$$MMT_s(3g,3A) = -C_q \frac{q^{3g+1}}{\zeta_q(2)} \frac{q^{(1-3s)A}}{1-q^{3s-1}} + O_\epsilon\left( q^{(1+3\epsilon)g + 3(1-s)A + 3 - 2s} +  q^{(\frac{4}{3}-s)3g+1}  E_s(3g,3A) \right)$$
while if $s=\frac{1}{3}$, we have
$$MMT_s(3g,3A) = \frac{C_q q^{3g+1} }{\zeta_q(2)} \left( g-A -\frac{1}{3} + \frac{1}{q-1} +\sum_{P} \frac{\deg(P)}{|P|^3+2|P|^2-1}\right)  + O_\epsilon \left(  q^{(1+\epsilon)g + (2+\epsilon)A} +q^{3g+1} \right)$$
\end{lem}

\begin{proof}

We first note that we can rewrite the $m$-sum in $MMT_s(3g,3A)$ as
\begin{align}\label{series_to_compute}
\mathcal{S} :=    \sum_{m\leq g-A-1} \sum_{f_1f_3^3\in \cM_{3m}} \frac{1}{|f_1|^{\frac{4}{3}-s}|f_3|^{3(1-s)}} \prod_{P|f_3}\left(1+\frac{1}{|P|+1}\right)^{-1} .
\end{align}
We consider the generating series
\begin{align*}
    \mathcal{D}(v) & := \sum_{f_1,f_3\in\cM}   \frac{1}{|f_1|^{\frac{4}{3}-s}|f_3|^{3(1-s)}}  \prod_{P|f_3}\left(1 +\frac{1}{|P|+1}\right)^{-1}v^{\deg f_1f_3^3}\\
    & = \left(\sum_{f_1\in \cM} \frac{v^{\deg f_1}}{|f_1|^{\frac{4}{3}-s}}\right) \left( \sum_{f_3\in \cM} \prod_{P|f_3}\left(1 +\frac{1}{|P|+1}\right)^{-1} \frac{v^{3\deg f_3} }{|f_3|^{3(1-s)}}\right) ,
\end{align*}
and we compute
$$\sum_{f_1\in \cM}\frac{v^{\deg f_1}}{|f_1|^{\frac{4}{3}-s}} = \sum_{f_1\in \cM} \left(\frac{v}{q^{\frac{4}{3}-s}}\right)^{\deg f_1} = \zeta_q\left(\frac{v}{q^{\frac{4}{3}-s}}\right) = \frac{1}{1-q^{s-\frac{1}{3}}v}  $$
and
\begin{align*}
    \sum_{f_3\in \cM} \prod_{P|f_3}\left(1 +\frac{1}{|P|+1}\right)^{-1} \frac{v^{3\deg f_3}}{|f_3|^{3(1-s)}} & = \prod_P\left(1 + \left(1 +\frac{1}{|P|+1}\right)^{-1} \frac{\left(\frac{v}{q^{1-s}}\right)^{3\deg P}}{1-\left(\frac{v}{q^{1-s}}\right)^{3\deg P}} \right) \\
    & =\prod_P\left(1-\left(\frac{v}{q^{1-s}}\right)^{3\deg P}\right) \prod_P \left(1-\frac{v^{3\deg P}}{|P|^{3(1-s)}|P|+2}\right)\\
    & = \zeta_q\left(\frac{v^3}{q^{3(1-s)}}\right) \prod_P \left(1-\frac{v^{3\deg P }}{|P|^{3(1-s)}(|P|+2)}\right) .
\end{align*}
Let $$\cK_s(v) :=  \prod_P \left(1-\frac{v^{3\deg P}}{|P|^{3(1-s)}(|P|+2)}\right)$$
which is an analytic function on the region $|v|< q^{1-s}$. Then 
\begin{align} \label{def-DV} \mathcal{D}(v) = \frac{\cK_s(v)}{(1-q^{3s-2}v^3) (1-q^{s-\frac{1}{3}}v)}, \end{align}
can be meromorphically extended to the region $|v|\leq q^{1-s-\epsilon}$ with poles at $v=q^{\frac{1}{3}-s}$ and $v=\xi_3^jq^{\frac{2}{3}-s}$ for $j=0,1,2$ where $\xi_3$ is a primitive root of unity. Notice
that $\cK_s(v)$ is uniformly bounded for  $|v|\leq q^{1-s-\epsilon}$.

Using Proposition \ref{perron},  and moving the integral to $\Gamma_4 = \{v : |v|=q^{1-s-\epsilon}\}$, where $s<1-\epsilon$, where $s\not=\frac{1}{3},\frac{2}{3}$, we have (where $\mathcal{S}$ is defined by \eqref{series_to_compute})
\begin{align*}
\mathcal{S} &= \frac{1}{2\pi i} \oint_{|v|=q^{-2}} \frac{\mathcal{D}(v)}{1-v^3}\frac{dv}{v^{3(g-A-1)+1}}  \\
&=  -  \Res_{v=q^{\frac{1}{3}-s}} \left(\frac{\mathcal{D}(v)}{(1-v^3)v^{3(g-A-1)+2}} \right) \\
  & \quad  - \sum_{j=0}^2 \left(\Res_{v=\xi_3^jq^{\frac{2}{3}-s}}\left(\frac{\mathcal{D}(v)}{(1-v^3)v^{3(g-A-1)+1}} \right) - \Res_{v=\xi_3^j}\left(\frac{\mathcal{D}(v)}{(1-v^3)v^{3(g-A-1)+1}} \right)  \right)\\
  & \quad + \frac{1}{2\pi i} \oint_{\Gamma_4} \frac{\mathcal{D}(v)}{1-v^3}\frac{dv}{v^{3(g-A-1)+1}} .
\end{align*}
The integral over $\Gamma_4$ is bounded by 
\begin{align*}
q^{(s-1+\epsilon)(3(g-A-1)+1)} \frac{q^{1-s-\epsilon}}{1-q^{3(1-s-\epsilon)}} \ll_{s, \epsilon} q^{(s-1+\epsilon)(3(g-A-1)+1)},
\end{align*}
where the bound does not depend on $q$.  Multiplying by 
\begin{align}\label{multiply_by}
\frac{q^{(\frac{4}{3}-s)3g+1}}{\zeta^2_q(2)} \prod_P \left(1 - \frac{1}{(|P|+1)^2}\right),
\end{align}
the contribution of the integral over $\Gamma_4$ to $MMT_s(3g,3A)$  is bounded by
\begin{align} \nonumber &q^{(\frac43 - s) 3g + 1} q^{(s-1+\epsilon)(3g-3A-2)}  \prod_P \left(1 - \frac{1}{(|P|+1)^2}\right) \left(1 - \frac{1}{|P|^2}\right)^2
\\  \label{ET-Gamma4} & \quad \ll q^{(1+3\epsilon)g} q^{3(1-s)A} q^{3-2s}
\end{align}
since the Euler product is absolutely bounded (independently of $q$). This gives the first error term for $MMT_s(3g, 3A)$.

It remains to compute the contribution to $MMT_s(3g, 3A)$ of the other residues. For $v=q^{\frac{1}{3}-s}$, we have
\begin{align*}
    \Res_{v=q^{\frac{1}{3}-s}} \left(\frac{\mathcal{D}(v)}{(1-v^3)v^{3(g-A-1)+1}} \right)  & = \lim_{v=q^{\frac{1}{3}-s}} \left(\frac{-q^{\frac{1}{3}-s} \cK_s(v) }{(1-q^{3s-2}v^3)(1-v^3)v^{3(g-A-1)+1}} \right) \\
    & = -\frac{\cK_s(q^{\frac{1}{3}-s})}{1-q^{-1}} \frac{q^{(3s-1)(g-A-1)}}{1-q^{1-3s}}.
\end{align*}
We now observe that $\zeta_q(2) = \frac{1}{1-q^{-1}}$ and
\begin{align*}
     \prod_P \left(1 - \frac{1}{(|P|+1)^2}\right) \cK_s(q^{\frac{1}{3}-s}) & =  \prod_P \left(1 - \frac{1}{(|P|+1)^2}\right) \left( 1 - \frac{1}{|P|^2(|P|+2)} \right) \\
     & = \prod_P \left(1-\frac{1}{|P|(|P|+1)}\right) = C_q .
\end{align*}
Therefore, multiplying  the residue at $v=q^{\frac{1}{3}-s}$ by \eqref{multiply_by}, 
the contribution of the residue to $MMT_s(3g, 3A)$ is
\begin{align*}
    C_q \frac{q^{(\frac{4}{3}-s)3g+1}}{\zeta_q(2)}\frac{q^{(3s-1)(g-A-1)}}{1-q^{1-3s}} = -C_q \frac{q^{3g+1}}{\zeta_q(2)} \frac{q^{(1-3s)A}}{1-q^{3s-1}} 
\end{align*}
which gives the main term of $MMT_s(3g, 3A)$ for $s \neq \frac{1}{3}$.

We now use the fact that $\cK_s(v)$ is invariant under multiplication of $v$ by cube roots of unity to get that the sum the residues at $v=\xi_3^jq^{\frac{2}{3}-s}$ will be 
\begin{align*} 
    & \frac{\cK_s(q^{\frac{2}{3}-s})}{1-q^{2-3s}}q^{(3s-2)(g-A -1)} \sum_{j=0}^2 \Res_{v=\xi_3^jq^{\frac{2}{3}-s}} \left(\frac{\xi_3^{-j}q^{s-\frac{2}{3}}}{(1-\xi_3^jq^{\frac{1}{3}})(1-q^{3s-2}v^3)}\right) \nonumber \\
   & =  \frac{\cK_s(q^{\frac{2}{3}-s})}{1-q^{2-3s}}q^{(3s-2)(g-A-1)} \sum_{j=0}^2 \lim_{v\to\xi_3^jq^{\frac{2}{3}-s}} \left(\frac{\xi_3^{-j}q^{s-\frac{2}{3}}(v-\xi_3^jq^{\frac{2}{3}-s})}{(1-\xi_3^jq^{\frac{1}{3}})(1-q^{3s-2}v^3)}\right) \nonumber \\
   &  =  - \frac{\cK_s(q^{\frac{2}{3}-s})}{1-q^{2-3s}}q^{(3s-2)(g-A-1)} \sum_{j=0}^2 \lim_{v\to\xi_3^jq^{\frac{2}{3}-s}} \left(\frac{1}{(1-\xi_3^jq^{\frac{1}{3}}) \prod_{i\not=-j} (1-q^{s-\frac{2}{3}} \xi_3^i v)}\right) \nonumber \\
  &   =  \frac{\cK_s(q^{\frac{2}{3}-s})}{1-q^{2-3s}}\frac{q^{(3s-2)(g-A-1)}}{q-1} 
\end{align*}
and, similarly, the sum over residues at $v=\xi_3^j$ will be 
\begin{align*}
     &  \frac{\cK_s(1)}{(1-q^{3s-2})} \sum_{j=0}^2 \Res_{v=\xi_3^j} \left(\frac{\xi_3^{-j}}{(1-q^{s-\frac{1}{3}}\xi_3^j)(1-v^3)}\right)  \\
     & = - \frac{\cK_s(1)}{(1-q^{3s-2})} \sum_{j=0}^2 \frac{1}{(1-q^{s-\frac{1}{3}}\xi_3^j) \prod_{i\not=-j} (1-\xi_3^i \xi_3^j)} \\
     &=  \frac{\cK_s(1)}{(1-q^{3s-2})(1-q^{3s-1})}.
\end{align*}

Furthermore, we see that $\cK_s(q^{\frac{2}{3}-s})$ is independent in $s$ so that for any $s\not=\frac{2}{3}$, we have that multiplying 
the sum of the residues at $v=\xi_3^jq^{\frac{2}{3}-s}$ by
\eqref{multiply_by}, we get a contribution to $MMT_s(3g,3A)$ bounded by
$$q^{(\frac{4}{3}-s)3g+1}q^{(3s-2)(g-A-1)} \ll q^{(\frac{4}{3}-s)3g+1}  \begin{cases}  1 & s<2/3  \\ q^{(3s-2)(g-A)} & s>2/3 \end{cases}.$$
Similarly, for $s<1-\epsilon$, we have that $\cK_s(1)$ is bounded and so the contribution from the residues at $v=\xi_3^j$ will be bounded, and multiplying by \eqref{multiply_by}, 
this gives a contribution bounded by $q^{(\frac{4}{3}-s)3g+1}$. This accounts for the second error term in $MMT_s(3g,3A)$ for $s\not=\frac{1}{3},\frac{2}{3}$, and this completes the proof in those cases.

For $s=\frac{2}{3}$, the residue at $v=q^{s-\frac{1}{3}}$ remain the same but we now get double poles at $v=\xi_3^j$. In this case we get
\begin{align*}
    \sum_{j=0}^2 \Res_{v=\xi_3^j}\left(\frac{\mathcal{D}(v)}{(1-v^3)v^{3(g-A-1)+1}} \right) & = \sum_{j=0}^2\lim_{v\to \xi_3^j} \frac{d}{dv} \left( \frac{\cK_{\frac{2}{3}}(v)(v-\xi_3^j)^2}{(1-q^{\frac{1}{3}}v)(1-v^3)^2v^{3(g-A-1)+1}} \right)\\
    &\ll g-A.
\end{align*}

For $s=\frac{1}{3}$, the residues at $v=\xi_3^jq^{s-\frac{2}{3}}$, $j=0,1,2$ and $\xi_3^j$, $j=1,2$ remain the same but we now have a a double pole at $v=1$ with residue 
\begin{align*}
    \Res_{v=1}\left(\frac{\mathcal{D}(v)}{(1-v^3)v^{3(g-A-1)+1}} \right) & = \frac{d}{dv} \left. \left( \frac{\cK_{\frac{1}{3}}(v)}{(1-q^{-1}v^3)(1+v+v^2)v^{3(g-A-1)+1}} \right)\right\vert_{v=1}\\
    & = -\frac{\cK_{\frac{1}{3}}(1)}{1-q^{-1}}\left(  g-A - \frac{1}{3} +  \frac{1}{q-1} - \frac{1}{3}\frac{\cK_{\frac{1}{3}}'(1)}{K_{\frac{1}{3}}(1)}\right)
\end{align*}

Multiplying by \eqref{multiply_by}, we get a contribution of
\begin{align*}
    - \frac{C_q q^{3g+1} }{\zeta_q(2)} \left( g-A -\frac{1}{3} + \frac{1}{q-1} - \frac{1}{3}\frac{\cK_{\frac{1}{3}}'(1)}{\cK_{\frac{1}{3}}(1)}\right) ,
\end{align*}
and we finish the proof with the observation that
\begin{align*}
    \frac{K_{\frac{1}{3}}'(v)}{K_{\frac{1}{3}}(v)} = \frac{d}{dv} \log K_{\frac{1}{3}}(v) = -3\sum_{P} \frac{v^{3\deg P -1} \deg P }{|P|^3+2|P|^2-v^{3\deg P}}
\end{align*}

\end{proof}

We can now combine Lemmas \ref{D_ETBound}, \ref{D_METBound} and \ref{MMTlem} to prove Proposition \ref{DualComp}.

\section{Proof of Main Theorems} \label{proof-thms}

\subsection{Proof of Theorem \ref{MainThm}}

By \eqref{AFEApplied} and Propositions \ref{PrincComp} and \ref{DualComp}, we have for any $A=cg$ with $0<c<1$, and $\epsilon < s < 1-\epsilon$ but  $s \neq \frac13$ that
\begin{align*}
   & \frac{1}{|\fH(3g)|} \sum_{\chi\in \fH(3g)} L(s,\chi) = \frac{\mathcal{P}_s(3g,3A) + \mathcal{D}_s(3g,3A)}{|\fH(3g)|} \\
 &\hspace{1cm}  =  M_q(s)+ C_q \frac{q^{(1-3s)A}}{1-q^{3s-1}} 
+ O_\epsilon\Big(q^{(\epsilon-3s)A} + q^{ A - (1-\epsilon)(3g+1)} + \frac{q^{ 3 (1-s+\epsilon)A}} {q^{\frac{3g}{2}}}\Big) \\
 &\hspace{1cm}  -C_q \frac{q^{(1-3s)A}}{1-q^{3s-1}} + O_\epsilon \left(   \frac{q^{(\frac34+\epsilon)g}}{q^{3A(s + \frac14+\epsilon)}}
+  \frac{q^{3(1-s)A+2-2s}}{q^{(2-3\epsilon)g}} +  q^{(1-3s)g} E_s(3g, 3A) \right) \\
 &\hspace{1cm}  =  M_q(s)
+ O_\epsilon\Big(  \frac{q^{ 3 (1-s+\epsilon)A}} {q^{\frac{3g}{2}}} 
+ \frac{q^{(\frac34+\epsilon)g}}{q^{3A(s + \frac14+\epsilon)}} + \frac{q^{3(1-s)A+2-2s}}{q^{(2-3\epsilon)g}} +
q^{(1-3s)g} E_s(3g, 3A) 
\Big)
\end{align*}
Optimizing the error term, we chose the cut off that equalizes the first two term error terms. Hence, we chose $A=\left\lfloor\frac{3}{5}g\right\rfloor$ and conclude that
for $s\not=\frac{1}{3}$
$$\frac{1}{|\fH(3g)|} \sum_{\chi\in \fH(3g)} L(s,\chi) = M_q(s) + O_{\epsilon} \left( q^{\frac{3}{10}(1-6s+\epsilon)g} +  q^{-\frac15 (1 + 9s - 15 \varepsilon )g + 2-2s} + q^{(1-3s)g}E_s(g) \right)$$
In the case $s=\frac{1}{3}$, we apply Propositions \ref{PrincComp} and \ref{DualComp} with $A=\left\lfloor\frac{3}{5}g\right\rfloor$ which gives 
\begin{align*}
 \frac{1}{|\fH(3g)|} \sum_{\chi\in\fH(3g)} L({\textstyle{\frac13}},\chi) &=  C_q (g + B_q) +  O (1)
\end{align*}
where $O(1)$ does not depend on $q$ or $g$ as long as $g \geq 2$, and 
\begin{align} \label{def-Bq}
B_q =   \frac{2}{3} - \frac{1}{q-1} + \sum_P \frac{\deg(P)}{|P|^2+|P|-1} + \sum_P \deg(P)\frac{|P|+2}{|P|^3+2|P|^2-1}.
\end{align}


\subsection{Proof of Theorem \ref{RMTThm}}

We see from the definition of $C_q$ that $\lim_{q\to\infty} C_q=1$ and by the prime polynomial theorem we get
\begin{align*}
\sum_P \deg(P) \frac{|P|+2}{|P|^3+2|P|^2-1} \ll \sum_{n=1}^\infty \frac{1}{q^{n}} = \frac{1}{q-1},
\end{align*}
and similarly for  $\sum_P \frac{\deg(P)}{|P|^2+|P|-1}$. This implies that $\lim_{q\to\infty} B_q= \frac{2}{3}$, and 
 we get from Theorem \ref{MainThm} that for $g$ fixed, 
\begin{align} \label{left-side}
    \lim_{q\to\infty} \frac{1}{|\fH(3g)|} \sum_{\chi \in \fH(3g)} L({\textstyle{\frac13}},\chi) = g + \tfrac23 + O(1) = g+O(1),
\end{align}
and it remains to compute the matrix integral 
$$
\int_{U(3g)} \deg(1-U) \overline{\det(1-\wedge^3U)} dU
$$
to see that it matches.

First we need to write the functions $\det(1-U)$ and $\det(1-\wedge^3U)$ in a common basis. For any infinite tuple $(\lambda_j)$ of non-negative integers with only finitely many non-zero entries we may write the partition
\begin{align}\label{partition}
\lambda = \prod_{j=1}^{\infty} j^{\lambda_j}
\end{align}
consisting of $\lambda_j$ copies of $j$.

The Newton identities then tells us for any $U\in U(N)$, we can write
$$\det(1-U) = \sum_{|\lambda|\leq N} \frac{(-1)^{\ell(\lambda)}}{z_{\lambda}} P_{\lambda}(U)$$
where $\ell(\lambda) = \sum \lambda_j$ is the length of $\lambda$ and 
$$ P_{\lambda}(U) = \prod_{j=1}^{\infty} \Tr(U^j)^{\lambda_j} \mbox{ and } z_{\lambda} = \prod_{j=1}^\infty j^{\lambda_j}\lambda_j!.$$
A result of Diaconis and Evans then shows that these $P_{\lambda}(U)$ form an orthogonal basis. 

\begin{thm}[Theorem 2.1 from \cite{DE}] \label{Diaconis-Evans}
    For any partitions $\lambda,\mu$ with $\min(|\lambda|,|\mu|)\leq N$, we have
    $$\int_{U(N)} P_\lambda(U) \overline{P_{\mu}(U)}dU = \delta_{\lambda \mu} z_{\lambda}$$
    where $\delta_{\lambda\mu}$ is the indicator function $\lambda = \mu$.
\end{thm}

So it remains to write $\det(1-\wedge^3U)$ in the basis formed by $P_{\lambda}(U)$. For this, we will need some new notation.

If $\lambda,\mu$ are partitions, written as in \eqref{partition}, then define the product  $\lambda\cdot\mu$ as
$$\lambda\cdot\mu = \prod_{j=1}^{\infty} j^{\lambda_j+\mu_j},$$
and by repeating this process, we can define positive integer powers of partitions
$$\lambda^k = \prod_{j=1}^{\infty} j^{k\lambda_j}.$$
We denote the the zero partition as 
$$\mathbb{0} = \prod_{j=1}^{\infty} j^0$$
and so we obtain $\lambda\cdot \mathbb{0} = \lambda$ and $\lambda^0 = \mathbb{0}$.
Finally, for any positive integer $k$, we then define 
$k\lambda = \prod_{j=1}^{\infty} (kj)^{\lambda_j}$
the partition consisting of $\lambda_j$ copies of $kj$. 

Now, let $P_3$ denote the set of partitions of $3$ and consider a tuple of non-negative integers $(a_{j\mu})_{j,\mu}$ where $\mu$ runs over the elements of $P_3$, $j$ runs over the positive integers and only finitely many of the $a_{j\mu}$ are non-zero. Then, for any such tuple we define the partition
$$\lambda(a_{j\mu}) = \prod_{j=1}^\infty \prod_{\mu\in P_3} (j\mu)^{a_{j\mu}}$$
Section 5.2 of \cite{M2} shows that we can write
$$\det(1-\wedge^3U) = \sum_{(a_{j\mu})_{j,\mu} }  \left[\prod_{j=1}^{\infty} \prod_{\mu\in P_3}\frac{1}{a_{j\mu}!} \left(\frac{(-1)^{\ell(\mu)}}{jz_\mu}\right)^{a_{j\mu}}\right] P_{\lambda(a_{j\mu})}(U).$$
Now, combining these facts and using Theorem \ref{Diaconis-Evans}, we obtain
\begin{align*}
    & \int_{U(N)} \det(1-U)\overline{\det(1-\wedge^3U)}dU \\
    = & \sum_{\substack{(a_{j\mu})_{j,\mu} \\ \lambda} } \left[ \prod_{j=1}^{\infty} \prod_{\mu\in P_3}\frac{1}{ a_{j\mu}!} \left(\frac{(-1)^{\ell(\mu)}}{j z_\mu}\right)^{a_{j\mu}} \right] \frac{(-1)^{\ell(\lambda)}}{z_{\lambda}}\int_{U(N)}  P_{\lambda}(U)\overline{P_{\lambda(a_{j\mu})}(U)} dU\\
    = & \sum_{\substack{(a_{j\mu})_{j,\mu}\\ |\lambda(a_{j\mu})|\leq N}} \prod_{j=1}^{\infty} \prod_{\mu\in P_3}\frac{1}{ a_{j\mu}!} \left(\frac{1}{jz_\mu}\right)^{a_{j\mu}}
\end{align*}
where we have used the fact that if $\lambda = \lambda(a_{j\mu})$ then 
$$\ell(\lambda) = \sum_{j=1}^{\infty} \sum_{\mu\in P_3} a_{j\mu}\ell(\mu).$$
To compute the sum over the $(a_{j\mu})_{j,\mu}$, we consider the generating series
\begin{align*}
    F(x) := \sum_{(a_{j\mu})_{j,\mu}} \prod_{j=1}^{\infty} \prod_{\mu\in P_3}\frac{1}{ a_{j\mu}!} \left(\frac{1}{jz_\mu}\right)^{a_{j\mu}} x^{|\lambda(a_{j\mu})|}.
\end{align*}
We use the fact that
$$|\lambda(a_{j\mu})| = 3 \sum_{j,\mu} ja_{j\mu}$$
to write
\begin{align*}
    F(x) & = \sum_{(a_{j\mu})_{j,\mu}} \prod_{j=1}^{\infty} \prod_{\mu\in P_3}\frac{1}{ a_{j\mu}!} \left(\frac{x^{3j}}{j z_\mu}\right)^{a_{j\mu}} \\
    & = \prod_{j=1}^{\infty} \prod_{\mu\in P_3} \left[\sum_{a_{j\mu}=0}^{\infty} \frac{1}{ a_{j\mu}!} \left(\frac{x^{3j}}{j z_\mu}\right)^{a_{j\mu}}\right] \\
    & = \prod_{j=1}^{\infty} \prod_{\mu\in P_3} \exp\left[ \frac{x^{3j}}{jz_\mu} \right]  = \exp\left(\sum_{j=1}^{\infty} \sum_{\mu\in P_3} \frac{x^{3j}}{jz_\mu}\right) \\
    & = \frac{1}{1-x^3} = \sum_{n=0}^{\infty} x^{3n}
\end{align*}
where we used the well-known fact that for any $r$
$$\sum_{\lambda \in P_r} \frac{1}{z_\lambda} = 1.$$
Hence, we may conclude that
$$\int_{U(N)} \det(1-U)\overline{\det(1-\wedge^3U)}dU = \sum_{3n\leq N} 1 = \left\lfloor \frac{N}{3} \right\rfloor +1.$$
and Theorem \ref{RMTThm} then follows from this and \eqref{left-side}, using $N = 3g$.

\end{document}